\documentclass[sn-mathphys,Numbered]{sn-jnl}% Math and Physical Sciences Reference Style
\usepackage{amsthm, mathtools,bbm,xcolor,soul}
\usepackage{graphicx}
\usepackage{subcaption}
\usepackage{hyperref}

\hypersetup{
    colorlinks,
    citecolor=black,
    filecolor=black,
    linkcolor=blue,
    urlcolor=black
}
\usepackage{graphicx}%
\usepackage{multirow}%
\usepackage{amsmath,amssymb,amsfonts}%
\usepackage{amsthm}%
\usepackage{mathrsfs}%
\usepackage[title]{appendix}%
\usepackage{xcolor}%
\usepackage{textcomp}%
\usepackage{manyfoot}%
\usepackage{booktabs}%
\usepackage{algorithm}%
\usepackage{algorithmicx}%
\usepackage{algpseudocode}%
\usepackage{listings}%
\usepackage{comment}
\usepackage{tikz-cd}
\usetikzlibrary{shapes.geometric}
\theoremstyle{definition}
\newtheorem*{defn*}{Definition}

\newtheorem*{rem*}{Remark}

\newtheorem{thm}{Theorem}[section]

\theoremstyle{definition}

\newtheorem{rem}[thm]{Remark}

\numberwithin{equation}{section}

\newtheorem{cor}[thm]{Corollary}

\newtheorem*{cor*}{Corollary}

\newtheorem{prop}[thm]{Proposition}

\newtheorem*{prop*}{Proposition}

\newtheorem{lem}[thm]{Lemma}

\newtheorem*{lem*}{Lemma}

\newcommand{\R}{\mathbb{R}}

\newcommand{\eR}{\overline{\mathbb{R}}}

\newcommand{\cZ}{\mathcal{Z}}

\newcommand{\cMtildestar}{\widetilde{\cM}_*}

\newcommand{\cQ}{\mathcal{Q}}

\newcommand{\N}{\mathbb{N}}

\newcommand{\cI}{\mathcal{I}}

\newcommand{\cTstar}{\mathcal{T}_*}

\newcommand{\cA}{\mathcal{A}}

\newcommand{\cX}{\mathcal{X}}

\newcommand{\E}{\mathbb{E}}

\def\P{\mathbb{P}}

\newcommand{\D}{\mathbf{D}}

\newcommand{\bC}{\mathbf{C}}

\newcommand{\cM}{\mathcal{M}}

\newcommand{\bM}{\mathbb{M}}

\newcommand{\0}{\mathbf{0}}
                
\newcommand{\SP}{\mathcal{S}}

\newcommand{\B}{\mathbb{B}}

\newcommand{\eqdist}{\stackrel{\text{d}}{=}}
\newcommand{\convdist}{\stackrel{\text{d}}{\longrightarrow}}
\begin{document}

\title[Article Title]{The Stationary Behavior of Reflecting Coupled Brownian Motions with Applications to 
Shortest Remaining Processing Time Queues}


%\author*[1]{\fnm{Sixian} \sur{Jin}}\email{sjin@csusm.edu}
\author*[1]{\fnm{Sixian} \sur{Jin}}\email{sjin@csusm.edu}

%\author*[1]{\fnm{Marvin} \sur{Pena}}\email{pena040@csusm.edu}
\author*[1]{\fnm{Marvin} \sur{Pena}}\email{pena051@csusm.edu}

\author*[1]{\fnm{Amber L.} \sur{ Puha}}\email{apuha@csusm.edu}


\affil[1]{\orgdiv{Department of Mathematics}, \orgname{California State University San Marcos}, \orgaddress{\street{333 S.\ Twin Oaks Valley Road}, \city{San Marcos}, \postcode{92096}, \state{CA}, \country{USA}}}

% \affil[2]{\orgdiv{Department}, \orgname{Organization}, \orgaddress{\street{Street}, \city{City}, \postcode{10587}, \state{State}, \country{Country}}}

% \affil[3]{\orgdiv{Department}, \orgname{Organization}, \orgaddress{\street{Street}, \city{City}, \postcode{610101}, \state{State}, \country{Country}}}


% \abstract{The abstract serves both as a general introduction to the topic and as a brief, non-technical summary of the main results and their implications. Authors are advised to check the author instructions for the journal they are submitting to for word limits and if structural elements like subheadings, citations, or equations are permitted.}
\abstract{
With the objective of characterizing the stationary behavior of the scaling limit for shortest remaining processing time (SRPT) queues with a heavy-tailed processing time distribution, as obtained in Banerjee, Budhiraja, and Puha (BBP, 2022), we study reflecting coupled Brownian motions (RCBM) $(W_t(a), a, t \geq 0)$. These RCBM arise by regulating coupled Brownian motions (CBM)
$(\chi_t(a), a,t \geq 0)$ to remain nonnegative. Here, for $t\geq 0$, $\chi_t(0)=0$ and
$\chi_t(a):=w(a)+\sigma B_t-\mu(a)t$ for $a>0$, $w(\cdot)$ is a suitable initial condition, $\sigma$ is a positive constant, $B$ is a standard Brownian motion, and $\mu(\cdot)$ is an unbounded, positive, strictly decreasing drift function. In the context of the BBP (2022) scaling limit, the drift function is determined by the model parameters, and, for each $a\geq 0$, $W_{\cdot}(a)$ represents the scaling limit of the amount of work in the system of size $a$ or less. Thus, for the BBP (2022) scaling limit, the time $t$ values of the RCBM describe the random distribution of the size of the remaining work in the system at time $t$.
Our principal results characterize the stationary distribution of the RCBM in terms of a maximum process $M_*(\cdot)$ associated with CBM starting from zero. We obtain an explicit representation for the finite-dimensional distributions of $M_*(\cdot)$ and a simple formula for its covariance. We further show that the RCBM converge in distribution to $M_*(\cdot)$ as time $t$ approaches infinity. From this, we deduce the stationary behavior of the BBP (2022) scaling limit, including obtaining an integral expression for the stationary queue length in terms of the associated maximum process. While its distribution appears somewhat complex, we compute the mean and variance explicitly, and we connect with the work of Lin, Wierman, and Zwart (2011) to offer an illustration of Little’s Law.
}

\keywords {Coupled Brownian Motions, Shortest Remaining Processing Time Queue, Regularly Varying Tails, Diffusion Approximation, Stationary Distribution, Nonstandard Scaling, Measure Valued State Process.\\
{\bf Acknowledgement:} Research supported in part by NSF Grant DMS-2054505.}


%\pacs[MSC Classification]{35A01, 65L10, 65L12, 65L20, 65L70}

\maketitle

\section{Introduction}\label{sec:Intro}

Let $B$ be a real valued standard Brownian motion, $\sigma$ be a finite positive constant and $\mu: (0,\infty)\to \mathbb (0,\infty)$ be a strictly decreasing, continuous function with $\lim_{a\to 0^+}\mu(a)=\infty$ and $\mu(\infty):=\lim_{a\to \infty}\mu(a)$.  Then $\mu(\infty)\ge 0$.
For $t\ge 0$, let
\begin{equation}\label{def:X}
X_t(a) = 
\begin{cases}
0,&\hbox{for }a=0,\\
\sigma B_t-\mu(a)t,&\hbox{for }a\in(0,\infty].
\end{cases}
\end{equation}
Then $((X_t(a), t\geq 0 ),a\in(0,\infty])$ are {\it coupled Brownian motions} (CBM) starting from zero, finite positive standard deviation $\sigma$, and continuous, strictly increasing negative drift that is bounded above by zero.
The drift diverges to minus infinity as $a$ decreases to zero, and so $X_\cdot(0)$ is set to be identically zero for convenience.  We refer to $\mu(\cdot)$ as the {\it drift function}.
For each $a\in[0,\infty]$, let
\begin{equation}\label{def:M*}
M_*(a):=%\lim_{t\to\infty}M_t(a)=
\sup_{t\geq 0}X_t(a).
\end{equation}
Then $M_*(0)=0$. Also, 
for each $a\in(0,\infty)$, $M_*(a)<\infty$ almost surely due to the negative drift of $X_\cdot(a)$ \cite{Harrison}.
In addition, $M_*(\infty)<\infty $ if and only if $\mu(\infty)>0$.
Here we study the distributional properties of the {\it maximum process} $(M_*(\cdot),M_*(\infty))$ and an associated finite nonnegative Borel measure $\cM_*$ that is defined below in \eqref{def:cM*0}, \eqref{def:cM*} and \eqref{def:cM*TotMass}.

The primary motivation for studying these stems from the connection between the maximum process $(M_*(\cdot),M_*(\infty))$ and the associated Borel measure $\cM_*$ and the measure valued scaling limit obtained in \cite{heavy tails} for shortest remaining processing time (SRPT) queues with heavy tailed processing time distributions.  The SRPT queue is a non-idling single server system that prioritizes preemptively serving the job with the shortest remaining service first.  SRPT queues are naturally of interest as they minimize the number of jobs in system over a large class of service disciplines \cite{Old Optimal, New Optimal}.  However, the size biased priority service scheme and preemptive nature of SPRT elicit complex dynamics that evade closed form analysis.  Hence, its scaling limits are of interest as tractable approximations.  There is an extensive literature that establishes such limits in a variety of settings, e.g., \cite{Atar, heavy tails, Down, Down_Sig, Gromoll, light tails, Lin, Puha}.   We describe the SRPT queue model, heavy tailed assumption, nature of the scaling, scaling limits and related results in more detail in Section \ref{sec:SRPT}.

The scaling limit obtained in \cite{heavy tails} is characterized in terms of a nonnegative space-time random field $(\widetilde{W}_t(a), a,t\geq 0)$ such that for each fixed $t\geq 0$, $\widetilde{W}_t(0)=0$, $\widetilde{W}_t:[0,\infty)\mapsto[0,\infty)$ is nondecreasing and $\widetilde{W}_t(\infty):=\lim_{a\to\infty}\widetilde{W}_t(a)<\infty$.  As detailed in Section \ref{sec:ht},
this random field is specified through a reflection mapping of CBM with standard deviation $\tilde\sigma$ and drift function $\widetilde{\mu}(\cdot)$ that are determined by the SPRT queue model parameters and with initial values that are finite, but not necessarily equal to zero.
Here we define a space-time random field $(W_t(a), a, t\geq 0)$ by applying the same reflection mapping to CBM with initial conditions that are also finite, but not necessary equal to zero (see \eqref{def:W0}), which allows us to consider a larger class of drift functions than those that arise in \cite{heavy tails}.  Our first main result states that the stochastic process $W_t(\cdot)$ converges in distribution to the stochastic process $M_*(\cdot)$ as $t\to\infty$ for any finite initial condition as defined in Section \ref{sec:DiffModel} (see Theorem \ref{thm:WStat} (i)).  In addition, if $\mu(\infty)>0$, then 
$(W_t(\cdot),W_t(\infty))$ converges in distribution to $(M_*(\cdot),M_*(\infty))$ as $t\to\infty$ (see Theorem \ref{thm:WStat} (i)).  By taking the standard deviation and drift function to be the ones associated with the SRPT queue scaling limit in \cite{heavy tails},
it follows that $(\widetilde{W}_t(\cdot),\widetilde{W}_t(\infty))$ converges in distribution to $(\widetilde{M}_*(\cdot),\widetilde{M}_*(\infty))$  as $t\to\infty$ for any finite initial condition, where $(\widetilde{M}_*(\cdot),\widetilde{M}_*(\infty))$ is defined as in \eqref{def:M*} with $\sigma=\tilde\sigma$ and $\mu(\cdot)=\widetilde{\mu}(\cdot)$ (see \eqref{def:SRPTmu} and \eqref{def:tildeM*}).  In particular, the distribution of $(\widetilde{M}_*(\cdot),\widetilde{M}_*(\infty))$ is the unique stationary distribution for the SRPT queue scaling limit $((\widetilde{W}_t(\cdot),\widetilde{W}_t(\infty)),t\ge 0)$.  See also Corollary \ref{cor:WStat}.

The one-dimensional distributions of $(M_*(\cdot),M_*(\infty))$ are known to be exponentially distributed.  Specifically, for each $a\in(0,\infty]$ such that $\mu(a)>0$, $M_*(a)$ is exponential distributed with rate parameter $2\mu(a)/\sigma^2$ (see \eqref{eq:M*(a)cdf}).  Another main result of this paper is Theorem \ref{thm:ZStat}, where a formula for the finite dimensional distributions of $(M_*(\cdot),M_*(\infty))$ is specified.  This formula expresses the $n$-dimensional distributions in terms of {the distributions of the maximums of two related Brownian motions, each with time varying drift.
Another way to view the result in Theorem \ref{thm:ZStat} is that it expresses the $n$-dimensional distributions in terms of certain $(n-1)$-dimensional distributions leading to a dimension reduction.  Using this formula with $n=2$ (see Theorem \ref{thm:ZStat_2d}) and some extensive calculations, we are able to find an explicit expression for the two point covariance function of $(M_*(\cdot),M_*(\infty))$ (see Corollary \ref{cor:MCov}).  Using the covariance expression in Corollary \ref{cor:MCov} and the known marginal distributions, we obtain a simple formula for the correlation.  Namely,  for $0< a_1\leq a_2\leq\infty$ such that $\mu(a_2)>0$,
\begin{equation}\label{eq:Corr}
\text{Corr}(M_*(a_1),M_*(a_2))=\frac{\mu(a_2)}{\mu(a_1)}\left(2-\frac{\mu(a_2)}{\mu(a_1)}\right),
\end{equation}
which is a quadratic function of the ratio of the drifts that is positive. The correlation tends to zero as $a_1$ decreases to zero for all fixed }$a_2\in(0,\infty]$ such that $\mu(a_2)>0$. Also, if $\mu(\infty)=0$, it tends to zero as $a_2$ tends to infinity for all fixed $a_1\in(0,\infty)$.  In Corollary \ref{cor:tildeMCov}, the covariance result is specialized to the SPRT scaling limit obtained in \cite{heavy tails} showing its dependence on the model parameters and the tail decay rate  in particular.

It is common for the SRPT queue state descriptor to be taken as the finite, nonnegative, Borel measure on the nonnegative real numbers that has a unit atom at the remaining processing time of each job in system at each time $t\geq 0$ (see \eqref{def:mvsd}).
This was done in \cite{heavy tails} and a scaling limit $\widetilde{\cZ}_{\cdot}$ for this process was also obtained (see \cite[Theorem 3]{heavy tails} and Section \ref{sec:ht} here).
This motivates our study of a nonnegative Borel measure $\cM_*$ that is associated with maximum process $M_*(\cdot)$. For this, we assume that
\begin{equation}\label{ass:mu}
\int_0^\infty \frac{1}{a^2\mu(a)} da<\infty.
\end{equation}
Then $\cM_*$ is defined as follows:
\begin{eqnarray}
\langle 1_{\{0\}},\cM_*\rangle&=&0,\label{def:cM*0}\\
\langle 1_{[0,a]},\cM_*\rangle&=&\int_0^a\frac{M_*(x)}{x^2}dx+\frac{M_*(a)}{a},\qquad\hbox{for }a\in(0,\infty),\label{def:cM*}\\
Z_*:=\langle 1_{[0,\infty)},\cM_*\rangle&=&\int_0^\infty\frac{M_*(x)}{x^2}dx.\label{def:cM*TotMass}
\end{eqnarray}
Here, for a Borel measurable set $A\subseteq[0,\infty)$,  $\langle 1_A,\cM_*\rangle$ denotes the measure of set $A$ under $\cM_*$. The integrals in \eqref{def:cM*} and \eqref{def:cM*TotMass} are well defined since $M_*(\cdot)$ is nonnegative and
continuous on $[0,\infty)$, although that $\cM_*$ is a finite nonnegative Borel measure requires verification.  This is done in Lemma \ref{lem:M*}.

Motivated by the application to SPRT queues with heavy tailed processing time distribution for each $t\geq 0$, we define the finite nonnegative Borel measure $\cZ_t$ by letting $\langle 1_{\{0\}},\cZ_t\rangle=0$ and using $W_t(\cdot)$ in place of $M_*(\cdot)$ in \eqref{def:cM*} and \eqref{def:cM*TotMass} to define $\langle 1_{[0,a]},\cZ_t\rangle$ for $a\in(0,\infty)$ and $Z_t:=\langle 1_{[0,\infty)},\cZ_t\rangle$ respectively
(see  \eqref{def:cZ0}, \eqref{def:cZ} and  \eqref{def:Zt}).  In Theorem \ref{thm:WStat} (ii), we establish convergence in distribution of $\cZ_t$ to $\cM_*$ as $t\to\infty$.  In addition in Theorem \ref{thm:WStat2}, we establish that moments of $W_t(a)$ converge to those for $M_*(a)$ as $t\to\infty$ for each $a\in[0,\infty]$ such that $\mu(a)>0$ and that, under suitable conditions on the drift function $\mu(\cdot)$, moments of $Z_t$ converge to those for $Z_*$ as $t\to\infty$. We recover convergence to stationarity properties of the SRPT queue measure valued scaling limit $\widetilde{\cZ}_{\cdot}$ obtained in \cite{heavy tails} as special cases of Theorems \ref{thm:WStat} and \ref{thm:WStat2} (see Corollaries \ref{cor:WStat} and \ref{cor:WStat2}).

A quantity of particular interest for the SPRT queueing model is the total queue length.
For the scaling limit arising in \cite{heavy tails}, this corresponds to $\widetilde{Z}_t:=\langle 1_{[0,\infty)},\widetilde{\cZ}_t\rangle$
for $t\geq 0$.  While the expression that is obtained in Corollary \ref{cor:WStat} for the stationary queue length, denoted by $\widetilde{Z}_*$, is explicit, i.e.,
$$
\widetilde{Z}_*:=\int_0^\infty\frac{\widetilde{M}_*(x)}{x^2}dx,
$$
its distribution appears to be complex due to the appearance of the maximum process $\widetilde{M}_*(\cdot)$ in the integrand.  By specializing the result in Theorem \ref{thm:ZStat_2d} to the SRPT queue setting and leveraging some extensive computations, we are able to compute explicit formulas for the mean and variance of $\widetilde{Z}_*$ (see Corollary  \ref{cor:tildeZstar}).  One conclusion from these calculations is that the mean does not equal the standard deviation and so $\widetilde{Z}_*$ is not exponentially distributed as it often happens for other single server queues.  Remark \ref{rem:tildeZstar} provides some additional observations regarding these moment formulas.  Finally, we relate our expression for the expected value of the stationary limiting queue length with the heavy traffic limit obtained in \cite{Lin} for suitably scaled mean stationary response times.  Under some mild conditions, we recover an illustration of Little's Law (see Remark \ref{rem:Lin} and \eqref{eq:Lin}).
 
The paper proceeds as follows.
In the next section, we introduce some of the basic notation to be used throughout the paper.  In Section \ref{sec:DiffModel}, we introduce what we refer to as reflecting coupled Brownian motion (RCBM), the model for which we determine the stationary behavior.  As described above, this is a slightly generalized version of the scaling limit that arises in \cite{heavy tails}.  The main results for RCBM are stated in Section \ref{sec:main}.  The two convergence results, Theorems \ref{thm:WStat} and Theorem \ref{thm:WStat2}, are stated in Section \ref{sec:Converge}.  Then, in Section \ref{sec:M*}, we state our results that characterize the finite dimensional distributions of the maximum process.  First, we do this for the two-dimensional distributions in Theorem \ref{thm:ZStat_2d}, which is leveraged to obtain the covariance formula in Corollary \ref{cor:MCov}.
Then, Theorem \ref{thm:ZStat} provides the formula for the $n$-dimensional distributions for all $n\geq 2$. In Section \ref{sec:SRPT}, we describe the SRPT queue model in detail, provide some background on scaling limits and describe the scaling and scaling limits that arise in the heavy tailed setting in \cite{heavy tails}.  Then, we apply our results to the scaling limits in \cite{heavy tails} to obtain and analyze its stationary behavior in Corollaries \ref{cor:WStat}, \ref{cor:tildeMCov}, \ref{cor:tildeZstar} and \ref{cor:WStat2} and Remarks \ref{rem:tildeZstar} and \ref{rem:Lin}. The proofs are given in Section \ref{sec:proofs}.

\subsection{Notation}\label{sec:notation}
We let  
$\N$ denote the set of positive integers,
$\R$ denote the set of real numbers, $\R_+$ denote the set of nonnegative real numbers and
$\overline\R_+:=\R_+\cup\{\infty\}$.
Given $a,b\in\R$, we let $a\wedge b$ (resp.\ $a\vee b$) denote the minimum (resp.\ maximum) of $a$ and $b$.
For $f:\R_+ \to \R$ and $T\in[0,\infty)$, we let $\|f\|_T := \sup_{0\leq t\leq T}\lvert f(t)\rvert$ and $\|f\|_\infty := \sup_{0\leq t< \infty}\lvert f(t)\rvert$.  We let $\bC(\R_+)$ denote the set of real-valued functions on $\R_+$ that are continuous and $\bC_b(\R_+)$ consist of those members of $\bC(\R_+)$ that are bounded.
For a Borel measurable subset $A$ of $\R_+$, the indicator function $1_A$ of $A$ is given by $1_A(\omega)=1$ if  $\omega\in A$ and $1_A(\omega)=0$ otherwise.
When $A=\R_+$, we denote $1_A$ by $1$.

In this paragraph, we fix an arbitrary Polish space $\SP$, i.e., a separable completely metrizable topological space.
For a collection $(Y_t,t\geq 0)$ of $\SP$-valued random elements and an $\SP$-valued random element $Y$, we write $Y_t\convdist Y$ as $t\to\infty$ to denote that $Y_t$ converges in distribution to $Y$ as $t\to\infty$.  When it may be helpful to clarify the Polish space $\SP$ in which the convergence takes place we write $Y_t\convdist Y$ as $t\to\infty$ in $\SP$.  Frequently, $\SP$ will be a product of finitely many Polish spaces endowed with the product topology.
A function $f:[0,\infty)\to\SP$ is an rcll function if it is right continuous and has left limits in $\SP$.
We let $\D([0,\infty),\SP)$ be the set of rcll functions of time taking values in $\SP$. We endow $\D([0,\infty),\SP)$ with Skorokhod $J_1$-topology.  With this topology, $\D([0,\infty),\SP)$ is also a Polish space.  All $\SP$-valued stochastic processes considered throughout are rcll with probability one.  We let $\bC([0,\infty),\SP)$ be the set of $f\in \D([0,\infty),\SP)$ that are continuous.  On $\bC([0,\infty),\SP)$, the Skorokhod $J_1$-topology and topology of uniform on compact sets are equivalent.

The one-dimensional Skorokhod map $\Psi$, which we define here, plays a significant role in our analysis. For this, we let $\D_0([0,\infty),\R)$ denote the space of all $f\in\D([0,\infty),\R)$ with $f(0)\geq0$. For $f\in \D_0([0,\infty),\R)$, $\inf_{0\leq s\leq t}f(s)$ is well defined for each $t\geq0$, and we define 
\begin{equation}\label{skorokhod}
\Psi[f](t):=f(t)-\inf_{0\leq s\leq t}f(s)\wedge 0,\quad t\geq0.
\end{equation}
Then $\Psi[f](\cdot)\in\D_0([0,\infty),\R_+)$. We make note that for $f,g \in \D_0([0,\infty),\R)$ and $T\geq 0$,
\begin{equation}\label{eq:Lip}
\left\Vert \Psi[f] - \Psi[g]\right\Vert_T\le 2\left\Vert f- g\right\Vert_T.
\end{equation}
Further details about the Skorokhod map can be found in \cite[Section 13.5]{Whitt}.

We let $\bM$ be the set of finite nonnegative Borel measures on $\R_+$.
Given $\zeta\in\bM$ and a Borel measurable function $g:\R_+\to\R$ that is integrable with respect to $\zeta$, we let $
\langle g,\zeta\rangle :=\int_{\R_+} g(x)\,\zeta(dx)$.
The set $\bM$ is endowed with the topology of weak convergence such that $\zeta_n\xrightarrow{w}\zeta$ as $n\to\infty$ if and only if $\lim_{n\to\infty}\langle g,\zeta_n\rangle=\langle g,\zeta\rangle$ for any $g\in\bC_b(\R_+)$. With this topology, $\bM$ is a Polish space \cite{Prokhorov}. In addition, if $\langle 1_{\{b\}},\zeta\rangle=0$ for all $b\in\R_+$, then $\zeta_n\xrightarrow{w}\zeta$ as $n\to\infty$ if and only if $\lim_{n\to\infty}\langle 1_{[0,b]},\zeta_n\rangle=\langle 1_{[0,b]},\zeta\rangle$ for all $b\in\R_+$ and $\lim_{n\to\infty}\langle 1,\zeta_n\rangle=\langle 1,\zeta\rangle$.   We let $\0$ denote the zero measure in $\bM$ and $\0(\cdot)$ denote the element of $\D([0,\infty),\bM)$ that is identically equal to $\0$.

%-----------------------Section 2----------------------------------------------------
\section{Reflecting Coupled Brownian Motions}\label{sec:DiffModel}

Let $\cI$ denote the set of continuous, nondecreasing stochastic processes
$w:\R_+\to\R_+$ that are independent of $B$ and satisfy that $w(0)=0$,  
\begin{equation}\label{def:W0}
\mathbb{E}\left[\int_0^1\frac{w(a)}{a^2}da\right]<\infty,
\end{equation}
and $w(\infty):=\lim_{a\to\infty}w(a)$ is such that $\E[w(\infty)]<\infty$. 
Then $\cI$ is the set of initial conditions for our model. 
Due to \eqref{def:W0}, for each $w\in\cI$, we have
\begin{equation}\label{eq:Z0}
	\int_0^\infty\frac{w(a)}{a^2}da<\infty,\qquad\hbox{almost surely.}
\end{equation}
As suggested by \eqref{eq:Z0}, we refer to $\cI$ as the finite initial conditions.
Note that the function $0(\cdot)$ that is identically equal to zero is an
element of $\cI$, and we will refer to this function as the zero initial condition.

Recall the definition of $B$, $\sigma$, $\mu(\cdot)$ and the CBM $((X_t(\cdot),X_t(\infty)), t\geq 0)$ starting from zero from \eqref{def:X}.
Let $w\in\cI$.  For $t\ge 0$ and $a\in\overline\R_+$, let
\begin{equation}\label{def:Wt}
\chi_t(a) :=  w(a)+X_t(a)\qquad\hbox{and}\qquad
W_t (a):=\Psi[\chi(a)](t).
\end{equation}
We refer to $((W_t(\cdot),W_t(\infty)),t\geq 0)$ as the {\it reflecting coupled Brownian motions} (RCBM) with standard deviation $\sigma$,
drift function $\mu(\cdot)$ and initial condition $w$.  When it might not be clear from context, we write 
$((W_t^w(\cdot),W_t^w(\infty)), t\geq 0)$ to clarify that the initial condition is $w$. 
In particular, $((W_t^0(\cdot),W_t^0(\infty)), t\geq 0)$ corresponds to a system with zero initial condition.
In Section \ref{sec:SRPT}, we explain how RCBM relate to
the heavy traffic scaling limit for a sequence of SRPT queues with a heavy tailed processing time distribution under a distribution dependent scaling (also see \cite[Theorem 1]{heavy tails}).

For each fixed $t\geq 0$, since $\chi_t:(0,\infty)\mapsto\mathbb{R}$ is continuous, strictly increasing and
$\lim_{a\to\infty}\chi_t(a)=\chi_t(\infty)$, it follows that $W_t:(0,\infty)\mapsto \R_+$ is nonnegative, continuous and nondecreasing and $\lim_{a\to\infty}W_t(a)=W_t(\infty)$.  In addition, $W_t(0)=0$ for all $t\geq 0$, $\lim_{a\to 0^+} W_0(a)=\lim_{a\to 0^+} w(a)=0$, and, for each $t>0$, $\lim_{a\to 0^+} \chi_t(a)=-\infty$ and so $\lim_{a\to 0^+} W_t(a)=0$.  Thus, for each fixed $t\geq 0$, $W_t(0)=0$, $W_t:[0,\infty)\mapsto \R_+$ is nonnegative, continuous and nondecreasing and $\lim_{a\to\infty}W_t(a)=W_t(\infty)$.  In fact, $(W_t(\cdot), t\geq 0)$ is a continuous time-space random field.

Assume that \eqref{ass:mu} holds and let $\cZ$ be the $\bM$ valued stochastic process that satisfies the following: for each $t\ge 0$,
\begin{eqnarray}
\langle 1_{\{0\}},\cZ_t\rangle&=&0 \label{def:cZ0} \\
\langle 1_{[0,a]},\cZ_t\rangle&=&\int_0^a\frac{W_t(x)}{x^2}dx+\frac{W_t(a)}{a},\qquad\hbox{for }a\in(0,\infty),\label{def:cZ}\\
Z_t:=\langle 1,\cZ_t\rangle&=&\int_0^\infty\frac{W_t(x)}{x^2}dx.\label{def:Zt}
\end{eqnarray}
We refer to $(\cZ_t,t\ge 0)$ as the {\it RCBM measure valued process}.
We remark that \eqref{eq:Z0} implies that $Z_0<\infty$ almost surely,
but it is not apriori clear that $Z_t<\infty$ for all $t>0$ and that $\cZ_t\in\bM$ for all $t\ge 0$.   This is established in Theorem \ref{thm:WStat} below.  When it might not otherwise be clear from context, we write $\cZ^w$ and $Z^w$ to clarify that the initial condition is $w$. 
In Section \ref{sec:SRPT}, we explain how the RCBM measure valued process $\cZ$ relates to a heavy traffic scaling limit for the measure valued state descriptors for a sequence of SRPT queues with a heavy tailed processing time distribution under distribution dependent scaling (also see \cite[Theorem 3]{heavy tails}).

Before proceeding to the main results of the paper that concern the stationary behavior of the RCBM and the RCBM measure valued process, we observe that they are point-recurrent at $0$ and $\0$ respectively.
\bigskip
\begin{prop}
\label{thm:hitszero} Let $w\in\cI$. Then
$
\mathbb{P}(\forall\ T\geq 0, \sup_{a\in\R_+}W_t^w(a)=0\hbox{ some }t\geq T)=1$.
If \eqref{ass:mu} also holds, then 
$\mathbb{P}(\forall\ T\geq 0,\ \cZ_t^w=\0\hbox{ some }t\geq T)=1$.
\end{prop}
\begin{proof} It suffices to prove the proposition for a deterministic initial state $w\in\cI$, which we assume henceforth. For $t\geq 0$, observe that
$\sup_{a\in\R_+}W_t^w(a)=W_t^w(\infty)=0$ if and only if $\chi_t(\infty)=\inf_{0\leq s\leq t} \chi_s(\infty)$.
The first result follows since $\chi_\cdot(\infty)$ is a Brownian motion with initial value $w(\infty)$ and nonpositive drift $-\mu(\infty)$, which is continuous and satisfies $\liminf_{t\to\infty}\chi_t(\infty)=-\infty$.
If \eqref{ass:mu} also holds, then since $\cZ_t^w=\0$ if and only if 
$\sup_{a\in\R_+}W_t^w(a)=W_t^w(\infty)=0$, the second result follows from the first.
\end{proof}

%------------Section 3-------------------
\section{Main Results}\label{sec:main}

In this section, we state our main results concerning the stationary behavior of the RCBM $((W_t^w(\cdot),W_t^w(\infty)),t\geq 0)$ and the RCBM measure valued process $(\cZ_t^w, t\geq 0)$ for $w\in\cI$.

\subsection{Convergence to the Maximum Process}\label{sec:Converge}

Here, we provide conditions under which the RCBM $((W_t^w(\cdot),W_t^w(\infty)),t\geq 0)$ and the RCBM measure valued process $(\cZ_t^w, t\geq 0)$ for $w\in\cI$ converge in distribution to the maximal process $(M_*(\cdot),M_*(\infty))$ and associated finite nonnegative Borel measure $\cM_*$ defined in \eqref{def:cM*0}, \eqref{def:cM*} and \eqref{def:cM*TotMass} respectively, as $t\to\infty$.

\bigskip
\begin{thm}\label{thm:WStat} Suppose $w\in\cI$.
\begin{itemize}
\item[(i)] As $t\to\infty$, $W_t^w(\cdot)\convdist M_*(\cdot)$ in $\bC(\mathbb{R}_+)$. 
In addition, if $\mu(\infty)>0$, then, as $t\to\infty$,
$(W_t^w(\cdot),W_t^w(\infty))\convdist(M_*(\cdot),M_*(\infty))$ in $\bC(\mathbb{R}_+)\times \mathbb{R}_+$.
\item[(ii)] Suppose that \eqref{ass:mu} also holds.
Then $\cZ_t^w\in\bM$ for all $t\geq 0$ and $\cM_*\in\bM$. Moreover, as $t\to\infty$,
 $(W_t^w(\cdot),\cZ_t^w)\convdist (M_*(\cdot),\cM_*)$ in $\bC(\mathbb{R}_+)\times \bM$.
In addition, if $\mu(\infty)>0$, then, as $t\to\infty$,
 $
 (W_t^w(\cdot),W_t^w(\infty),\cZ_t^w)\convdist (M_*(\cdot),M_*(\infty),\cM_*)$  in \ $\bC(\mathbb{R}_+)\times \mathbb{R}_+\times \bM$.
 \end{itemize}
\end{thm}

Theorem \ref{thm:WStat} is proved in Section \ref{sec:WStat}.
As an immediate consequence of Theorem \ref{thm:WStat}, it follows that the distribution of
the maximum process is the unique stationary distribution for RCBM with standard deviation
$\sigma$ and drift function $\mu(\cdot)$.  We discuss its distribution in Section \ref{sec:M*} below.
Next, for $w\in\cI$, we provide the limiting behavior as $t\to\infty$ of the moments of $W_t^w(a)$ for $a\in\overline\R_+$ and $Z_t^w$ in terms of the corresponding stationary moments. For this, recall that $Z_*=\langle 1, \cM_*\rangle$. 

\bigskip
\begin{thm}\label{thm:WStat2} Suppose $w\in\mathcal{I}$ and $\gamma\ge 1$.
\begin{itemize}
\item[(i)] For $a\in\overline\R_+$ such that $\mu(a)>0$ and $\E\left[\left(w(a)\right)^\gamma\right]<\infty$, $\lim_{t\to \infty} \E[(W_t^w(a))^\gamma]=\E[(M_*(a))^\gamma]$.
\item[(ii)] If \eqref{ass:mu} holds, then $\lim_{t\to\infty}\E[Z_t^w]=\E[Z_*]$.
Furthermore, if $\gamma>1$, $\E[\left(Z_0^w\right)^\gamma]<\infty$,
\begin{equation}
\label{eq:highmoment}
\int_0^1\frac{1}{x^{2\gamma}\mu(x)^\gamma}dx<\infty\quad\hbox{and}\quad
\int_1^\infty\frac{1}{x^{\gamma}\mu(x)^\gamma}dx<\infty,
\end{equation}
then $\lim_{t\to\infty}\E[(Z_t^w)^\gamma]=\E[(Z_*)^\gamma]$.
\end{itemize}
\end{thm}
The proof of Theorem \ref{thm:WStat2} is given at the end of Section \ref{sec:WStat}.

\subsection{Distributional Properties of the Maximum Process}
\label{sec:M*}

In this section, we study the finite dimensional distributions of the maximum process $(M_*(\cdot),M_*(\infty))$.   We obtain a formula for the $n$-dimensional distributions for $n\geq 2$. 
We begin in Section \ref{sec:1d} by summarizing a known result for the one dimensional distributions.
We proceed in Section \ref{sec:2d} by presenting the two-dimensional distributions.
This case provides insight into the principle structure of the multidimensional joint distribution as well as leading to
an explicit formula for the covariance function.  With this foundational understanding, we then address the more complex 
$n$-dimensional scenario in Section \ref{sec:nd}.

\subsubsection{One Dimensional Distributions}\label{sec:1d}
To begin, we recall an established result for the one dimensional
distributions of $(M_*(\cdot),M_*(\infty))$. For all $t\ge 0, \nu\ge 0$ and $x\in\R_+$,
it is known that
\begin{equation}\label{eq:CDFMax}
\P\left(\sup_{s\in[0,t]} \left(\sigma B_s-\nu s\right)\le x\right)
=
\Phi\left(\frac{x+\nu t}{\sigma\sqrt{t}}\right)-\exp\left(\frac{-2\nu x}{\sigma^2}\right)\Phi\left(\frac{-x+\nu t}{\sigma\sqrt{t}}\right),
\end{equation}
where $\Phi$ denotes the cumulative distribution function of a standard normal random variable (see Corollary 1.8.7 in \cite{Harrison} with $y=x$). Upon considering $a\in(0,\infty]$ and $\nu=\mu(a)\geq 0$ in \eqref{eq:CDFMax}, letting $t\to\infty$ 
and using the fact that $\mu(a)\geq 0$, we find that for each $x\in\R_+$
\begin{equation}\label{eq:M*(a)cdf}
\P\left(M_*(a)\le x\right)=1-\exp\left(-\frac{2\mu(a)}{\sigma^2}x\right).
\end{equation}
In particular, $M_*(a)$ is exponentially distributed with rate $2\mu(a)/\sigma^2$ for each $a\in(0,\infty]$.
If $\mu(\infty)=0$, then this is understood as a degenerate exponential distribution with rate $0$ and
$\P\left(M_*(\infty)\le x\right)=0$ for all $x\in\R_+$, i.e., $\P\left(M_*(\infty)=\infty\right)=1$.

\subsubsection{Two Dimensional Distributions}\label{sec:2d}
For $a\in\overline\R_+$ and $t\ge 0$, define
\begin{equation}\label{def:Mt}
M_t(a):=\sup_{0\leq s\le t} X_s(a).
\end{equation}
%\bigskip
\begin{thm}\label{thm:ZStat_2d}
Suppose that $0<a_1<a_2\leq \infty$ and $0\le x_1 < x_2<\infty$.  Then,
\begin{eqnarray}\label{eq:ZStat_2d}
&&\P\left( M_*(a_1)\le x_1, M_*(a_2)\le x_2\right)\nonumber\\
&&\qquad=\P\left( M_{\tau_1}(a_1)\le x_1\right)-\exp\left(\frac{-2\mu(a_2)x_2}{\sigma^2}\right)\P\left( V_{\tau_1}^*\le x_1\right),
\end{eqnarray}
where $\tau_1=\frac{ x_2- x_1 }{\mu(a_1)-\mu(a_2)}$ and $V_{\tau_1}^*:=\sup_{t\in[0,\tau_1] } (X_t(a_1)+2\mu(a_2)t)$.
\end{thm}

We remark that if $a_2=\infty$ and $\mu(a_2)=0$, then $\P\left( M_*(a_1)\le x_1, M_*(a_2)\le x_2\right)=0$ because $\P\left( M_*(a_2)\le x_2\right)=0$
and the right side of \eqref{eq:ZStat_2d} is also zero since $V_{\tau_1}^*=M_{\tau_1}(a_1)$ in this case.  We also note that
$0\le x_1 < x_2<\infty$ is considered in Theorem \ref{thm:ZStat_2d} since $M_{\tau_1}(a_1)\leq M_*(a_2)$ for all $0<a_1<a_2\leq \infty$.
In particular, if $0\le x_2\leq x_1<\infty$, then $\P\left( M_*(a_1)\le x_1, M_*(a_2)\le x_2\right)=\P\left( M_*(a_2)\le x_2\right)$, which is a one-dimensional distribution and is given by \eqref{eq:M*(a)cdf}.  
Theorem \ref{thm:ZStat_2d}, together with the more general $n$-dimensional version Theorem \ref{thm:ZStat} stated below, is proved  in Section \ref{sec:distriproof}.  Next, we make one additional observation concerning the formula \eqref{eq:ZStat_2d}.

\bigskip
\begin{rem}
\label{rem:2d}
Let $0<a_1<a_2\leq \infty$ and $0\le x_1 < x_2<\infty$.  Observe that
\begin{eqnarray*}
\{M_*(a_1)\le x_1, M_*(a_2)\le x_2\}
&=&\{X_s(a_1)\le x_1\hbox{ and }X_s(a_2)\le x_2\hbox{ for all }s\ge 0\}\\
&=&\{\sigma B_s\le\min\{x_1+\mu(a_1)s, x_2+\mu(a_2)s\}\hbox{ for all }s\ge 0\}.
\end{eqnarray*}
The two lines that appear in the minimum intersect at time $\tau_1$.  Since $x_1<x_2$, we have
\begin{eqnarray*}
\{M_*(a_1)\le x_1, M_*(a_2)\le x_2\}&=&\{ M_{\tau_1}(a_1)\leq x_1\hbox{ and }M_*(a_2)\leq x_2 \}\\
&=&\{ M_{\tau_1}(a_1)\leq x_1\} \setminus \{ M_{\tau_1}(a_1)\leq x_1\hbox{ and }M_*(a_2)>x_2 \}.
\end{eqnarray*}
Hence, an implication of the result in Theorem \ref{thm:ZStat_2d} is that, for $0<a_1<a_2\leq \infty$ and $0\le x_1 < x_2<\infty$,
$$
\mathbb{P}\left(M_{\tau_1}(a_1)\leq x_1 \mid M_*(a_2)>x_2\right)=\mathbb{P}\left( V_{\tau_1}^*\leq x_1\right).
$$
We can rewrite this identity by
\begin{eqnarray*}
&&\P(\sigma B_s\leq x_1+\mu(a_1)s\ \text{for all}\ s\leq \tau_1|\sigma B_s> x_2+\mu(a_2)s\ \text{for some}\ s\ge 0)\\
&&\qquad =\P(\sigma B_s\leq x_1+(\mu(a_1)-2\mu(a_2))s, \text{for all}\ s\leq \tau_1).
\end{eqnarray*}
A direct interpretation states that if a scaled Brownian motion \( \sigma B \) is known to ultimately pass through the line \( \ell_2(s) = x_2 + \mu(a_2)s \), $s\geq 0$, then the likelihood of keeping its position under a steeper line \( \ell_1(s)= x_1 + \mu(a_1)s \), $s\geq 0$, with the smaller starting point $x_1$ until time \( \tau_1 \) (the time of intersection of these two lines) is the same as the likelihood of keeping its position under the less steep line \( \ell_{1,2}(s) = x_1 + (\mu(a_1) - 2\mu(a_2))s \), $s\geq 0$, until time \( \tau_1 \).
\end{rem}

By leveraging the result in Theorem \ref{thm:ZStat_2d} and some intensive calculations shown in Section \ref{cor:MCov}, we explicitly represent the covariance structure of $M_*(\cdot)$ as follows.

\bigskip
\begin{cor}\label{cor:MCov} For $0<a_1<a_2\le\infty$ such that $\mu(a_2)>0$,
\begin{equation}\label{eq:MCov}
\hbox{Cov}(M_*(a_1),M_*(a_2))=\frac{\sigma^4}{4\mu(a_1)^2}\left(2-\frac{\mu(a_2)}{\mu(a_1)}\right).
\end{equation}
\end{cor}
Combining \eqref{eq:MCov} with \eqref{eq:M*(a)cdf}, we recover the correlation as previously noted in \eqref{eq:Corr}.
Another way to express the result in Corollary \ref{cor:MCov} is to write
\begin{eqnarray*}
\hbox{Cov}(M_*(a_1),M_*(a_2)-M_*(a_1))=\frac{\sigma^4}{4\mu(a_1)^2}\left(1-\frac{\mu(a_2)}{\mu(a_1)}\right),
\end{eqnarray*}
for $0<a_1<a_2\le\infty$ such that $\mu(a_2)>0$.  Indeed this is equivalent to Corollary \ref{cor:MCov} due to
\eqref{eq:M*(a)cdf} and bilinearity of covariance.

\subsubsection{Finite Dimensional Distributions}\label{sec:nd}

To characterize the $n$-dimensional joint distributions of $M_*$ for arbitrary $n\geq 3$,
we first introduce some notation and present some intuitive observations.  For $n\in\N$, let
\begin{eqnarray*}
\mathcal{A}_n&:=&
\{ \vec{a}\in\R_+^{n-1}\times\eR_+ : 0<a_1<a_2<\cdots<a_n\leq \infty\}.
\end{eqnarray*}
We wish to compute $\P\left( M_*(a_i)\le x_i\hbox{ for }i=1,2,\dots,n\right)$ for each $n\in\N$, $\vec{a}\in\mathcal{A}_n$ and $\vec{x}\in \R_+^n$. 
Given $n\in\N$ and $\vec{a}\in\mathcal{A}_n$, certain choices of $\vec{x}\in\R_+^n$ result in redundant constraints that lead to a dimension reduction.  For instance, similarly to the $n=2$ case above, if  $n\in\N$, $\vec{a}\in\mathcal{A}_n$, and $\vec{x}\in\R_+^n$ are such that there exists $i'\in\{1,2,\dots,n-1\}$ such that $x_{i'+1}\leq x_{i'}$, then, $M_*(a_{i'+1})\leq x_{i'+1}$ implies $M_*(a_{i'})\leq x_{i'}$ because of $M_*(a_{i'})\leq M_*(a_{i'+1})$. Therefore,
\begin{eqnarray*}
\P\left( M_*(a_i)\le x_i\hbox{ for }i=1,2,\dots,n\right)
&=&\P\left( M_*(a_i)\le x_i\hbox{ for }i\neq i'\right).
\end{eqnarray*}
However, it is not enough to restrict to $\vec{x}\in\R_+^n$ such that 
$0\leq x_1<x_2<\cdots< x_n<\infty$.
For example, consider a choice of $\vec{a}\in\mathcal{A}_3$ such that $\mu(a_1)=3$, $\mu(a_2)=2$ and $\mu(a_3)=1$,
and take $x_1=1$, $x_2=3$, and $x_3=4$.
Then
$$
\P\left( M_*(a_i)\le x_i\hbox{ for }i=1,2,3\right)=\P\left( \sigma B_s\le \min\{ 3s+1, 2s+3, s+4\}\hbox{ for all }s\ge 0\right).
$$
Simple algebra shows that $2s+3>\min\{ 3s+1, 2s+3, s+4\}$ for any $s\ge 0$.  See Figure \ref{fig:1}.  Hence,
\begin{eqnarray*}
\P\left( M_*(a_i)\le x_i\hbox{ for }i=1,2,3\right)&=&\P\left( \sigma B_s\le \min\{ 3s+1,  s+4\}\hbox{ for all }s\ge 0\right)\\
&=&\P\left( M_*(a_1)\le x_1, M_*(a_3)\le x_3\right),
\end{eqnarray*}
which is a 2-dimensional distribution.  To avoid such issues, the lines $\ell_i(s)=\mu(a_i) s + x_i$, $s\ge 0$ and $i=1,2,3$, must intersect sequentially.  For example, if $x_3=6$ rather than 4,  we see that $\ell_1$ and $\ell_2$ intersect at $s=2$ and $\ell_2$ and $\ell_3$ intersect at $s=3$.  Thus, we have
\begin{eqnarray*}
&&\P\left( M_*(a_i)\le x_i\hbox{ for }i=1,2,3\right)\\
&&\qquad=\P\left(\sup_{s\in[0,2)} X_s(a_1)\le 1, \sup_{s\in[2,3)}X_s(a_2)\le 3,\ \sup_{s\ge 3} X_s(a_3)\le 6 \right),
\end{eqnarray*}
which remains 3-dimensional. See Figure \ref{fig:2}.

\begin{figure}[h!]
    \centering
    \begin{subfigure}[b]{0.49\textwidth}
        \centering
        \includegraphics[width=\textwidth]{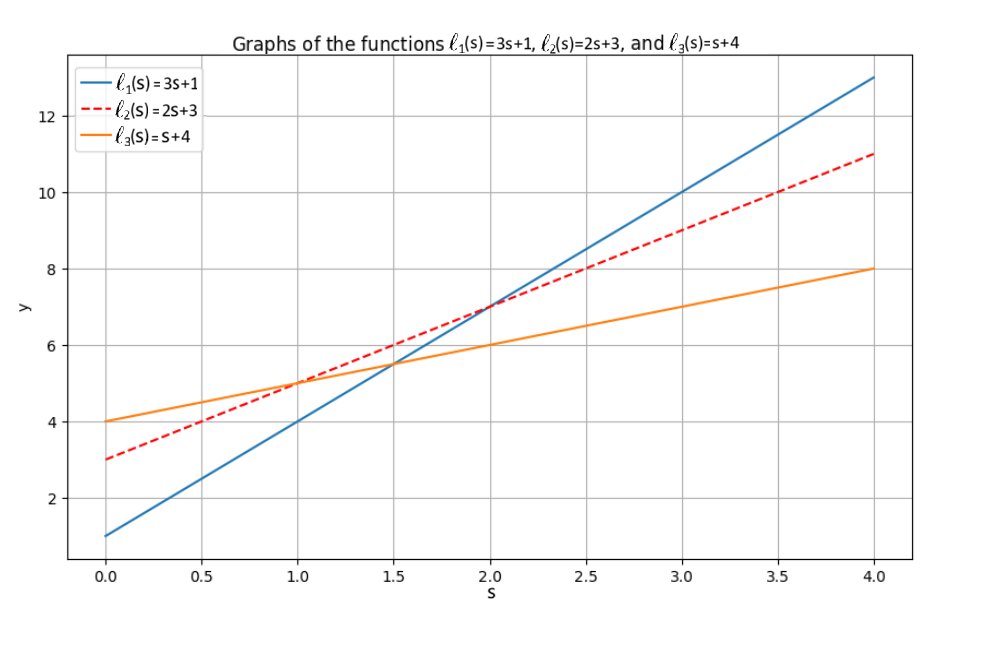}
        \caption{Redundant restriction $\ell_2(s)=2s+3$ causes a dimension reduction}
        \label{fig:1}
    \end{subfigure}
    \hfill
    \begin{subfigure}[b]{0.49\textwidth}
        \centering
        \includegraphics[width=\textwidth]{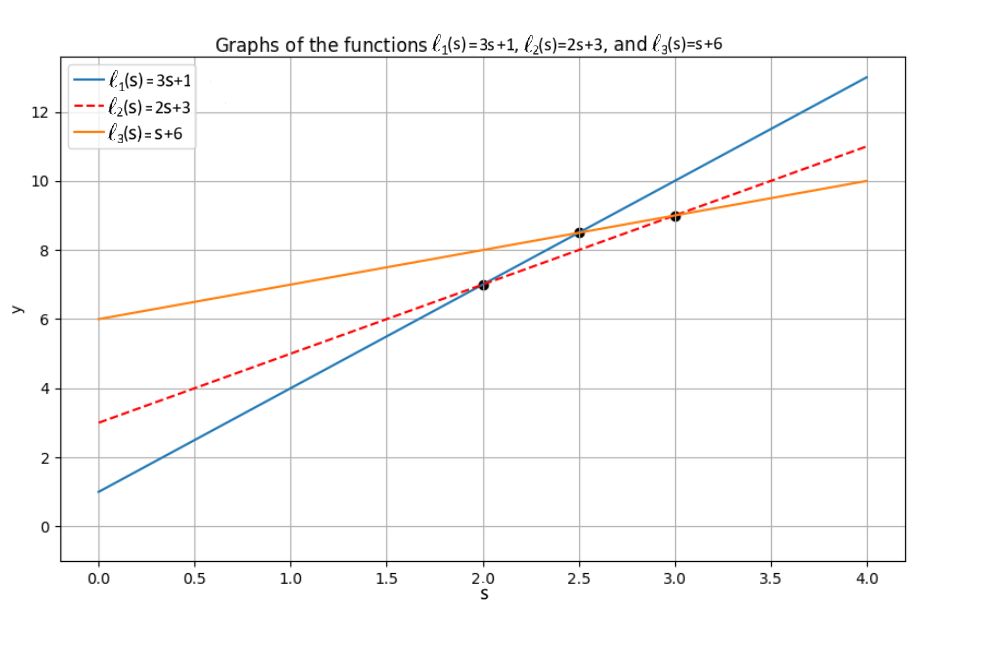}
        \caption{Valid restrictions result in a genuine 3-dimensional joint distribution}
        \label{fig:2}
    \end{subfigure}
    \label{fig:figures}
\end{figure}

In general, given $n\in\N$, $\vec{a}\in\mathcal{A}_n$ and $\vec{x}\in \R_+^n$, we have
\begin{eqnarray}
	&&\{M_*(a_i)\le x_i\hbox{ for }i=1,2,\dots,n\}\nonumber\\
&&\qquad=\{\sigma B_s\le \mu(a_i)s+x_i\hbox{ for all }s\ge 0\hbox{ and }i=1,2,\dots,n\}.\label{eq:const}\\
&&\qquad=\{\sigma B_s\le \min\{\ell_i(s) : i=1,2,\dots,n\}\hbox{ for all }s\ge 0\},\nonumber
\end{eqnarray}
where $\ell_i(s)=\mu(a_i)s+x_i$ for $s\ge0$ and $i=1,2,\dots,n$.
For $n\in\N$, $\vec{a}\in\mathcal{A}_n$ and $\vec{x}\in\R_+^n$, let $\tau_0=0$, for $i=1,2,\dots,n-1$, 
$$
\tau_i=\frac{x_{i+1}-x_i}{\mu(a_i)-\mu(a_{i+1})},
$$
and $\tau_n=\infty$.  Then, $\tau_i$ denotes the $s$-coordinate of the intersection time of lines $\ell_i$ and $\ell_{i+1}$, i.e., $\ell_i(\tau_i)= \ell_{i+1}(\tau_i)$ for $i=1,2,\dots,n-1$.
In order to guarantee that those lines intersect sequentially, we consider the following set: for $n\in\N$ and $\vec{a}\in\mathcal{A}_n$, let
\begin{equation}\label{ass:tau}
\mathcal{X}_n^{\vec{a}}:=\{ \vec{x}\in\R_+^n : 0=\tau_0<\tau_1<\cdots<\tau_{n-1}<\tau_n=\infty\}.
\end{equation}
For $n\in\N$, $\vec{a}\in\mathcal{A}_n$ and $\vec{x}\in\R_+^n$,
it follows from Lemma \ref{lem:tau}, which is stated and proved in Section \ref{sec:distriproof},
that all constraints in \eqref{eq:const} are necessary if and only if $\vec{x}\in\mathcal{X}_n^{\vec{a}}$.

To state our $n$-dimensional result, we introduce some additional stochastic processes to simplify the notation.
Let $n\in\N$, $\vec{a}\in\mathcal{A}_n$ and $\vec{x}\in\mathcal{X}_n^{\vec{a}}$.  For each $t\ge 0$,
let $i_t:= i\in\{1,\dots,n\} $ such that $ t\in(\tau_{i-1},\tau_{i}]$.  Then, for all $t\ge 0$,
let
\begin{alignat}{3}
U_t&:=\sigma B_t -\int_0^t\mu(a_{i_s}) ds&\qquad\hbox{and}\qquad U_t^*&:=\sup_{s\in[0,t]} U_s;\label{def:U}\\
V_t&:=U_t+2\mu(a_n) t &\qquad\hbox{and}\qquad V_t^*&:=\sup_{s\in[0,t]} V_s.\label{def:V}
\end{alignat}
With these, we can state our result characterizing the finite dimensional distributions of $(M_*(\cdot),M_*(\infty))$.

\bigskip
\begin{thm}\label{thm:ZStat}
Suppose that $n\in\N$, $\vec{a}\in\mathcal{A}_n$ and $\vec{x}\in\mathcal{X}_n^{\vec{a}}$.  Then, 
\begin{eqnarray*}
&&\P\left( M_*(a_i)\le x_i\hbox{ for }i=1,2,\dots,n\right)\\
&&\qquad=\P\left( U_{\tau_{n-1}}^*\le x_1\right)-\exp\left(\frac{-2\mu(a_n)x_n}{\sigma^2}\right)\P\left( V_{\tau_{n-1}}^*\le x_1\right).
\end{eqnarray*}
\end{thm}

The proof of Theorem \ref{thm:ZStat} is given in Section \ref{sec:distriproof}.  To better understand the intuition of this theorem as well as the processes of $U$ and $V$, we present the following observation similar to Remark \ref{rem:2d}. 

\bigskip
\begin{rem}\label{rem:nd}
In light of the fact $\exp\left(\frac{-2\mu(a_n)x_n}{\sigma^2}\right)=\P(M_*(a_n)>x_n)$, we can rewrite $\P\left( M_*(a_i)\le x_i\hbox{ for }i=1,2,\dots,n\right)$ under conditional probability as:
\begin{eqnarray*}
&&\P\left( M_*(a_i)\le x_i\hbox{ for }i=1,2,\dots,n\right)=\P\left( M_*(a_i)\le x_i\hbox{ for }i=1,2,\dots,n-1\right)\\
&&\qquad-\P\left(M_*(a_n)>x_n\right)\P\left( M_*(a_i)\le x_i\hbox{ for }i=1,2,\dots,n-1\mid M_*(a_n)>x_n\right).
\end{eqnarray*}
Thus, we can understand the following:
\begin{eqnarray*}
\P\left( U_{\tau_{n-1}}^*\le x_1\right)&=&\P\left( M_*(a_i)\le x_i\hbox{ for }i=1,2,\dots,n-1\right),\\
\P\left( V_{\tau_{n-1}}^*\le x_1\right)&=&\P\left( M_*(a_i)\le x_i\hbox{ for }i=1,2,\dots,n-1\mid M_*(a_n)>x_n\right).
\end{eqnarray*}
Intuitively, the first equality can be explained as the maximums of the CBM starting from zero are suitably controlled until time $\tau_{n-1}$ by considering a Brownian motion $U$ with a suitably defined piecewise constant drift that changes at the times $\tau_1, \tau_2,\dots,\tau_{n-1}$ and confining that process to be bounded above by $x_1$ until time $\tau_{n-1}$. At the same time, when conditioning on the event $\{ M_{\tau_n}(a_n)>x_n\}$, $2\mu(a_n)$ must be added to the each of the piecewise constant drifts, which gives the process $V$.
\end{rem}

\section{Application to SRPT Queues}\label{sec:SRPT}

As an application of the main results in this paper, we obtain the stationary behavior of the heavy traffic scaling limit obtained in \cite{heavy tails} for a sequence of SRPT queues with a heavy tailed processing time distribution.  In Section \ref{sec:SPRTQueueModel}, we describe the SPRT queue model and its measure valued state descriptor.  As noted in the introduction, the complex SRPT dynamics result in a system that cannot be analyzed in closed form.  Hence, scaling limits are of interest as tractable approximations.  In Section \ref{sec:ScalingLimits},  we briefly summarize some of the scaling limit theorems that have been established for SRPT queues, including introducing the distribution dependent scaling that gives rise to the scaling limit in \cite{heavy tails}.  In Section \ref{sec:ht}, we discuss this scaling limit and make note that it is a specific instance of the RCBM and RCBM measure valued process defined in Section \ref{sec:DiffModel}.  Finally, in Section \ref{sec:SRPTthms}, we state our main results for the stationary behavior of the scaling limit obtained in \cite{heavy tails}.

\subsection{The Model and Measure Valued State Descriptor}\label{sec:SPRTQueueModel}
Here we consider a single server SRPT queue.
For this, jobs arrive to the system according to a possibly delayed renewal process $E(\cdot)$ with jumps of size one, arrival rate $\lambda\in(0,\infty)$ and finite interarrival time standard deviation $\sigma_a$.
Upon its arrival, each job is assigned a processing time taken from an independent and identically distributed sequence $\{v_i\}_{i\in\N}$ of positive random variables with common cumulative distribution function (CDF) $F$ that has finite, positive mean $\mathbb{E}[v_1]$ and finite, positive standard deviation $\sigma_s$; the $i$th job to arrive is assigned processing time $v_i$.
There may also be a finite number $q_0$ of initial jobs in the system at time zero with positive remaining processing times $\{ v_{-q_0+1}, v_{-q_0+2},\dots, v_0\}$ satisfying $v_{-i-1}\le v_{-i}$ for $i=0,\dots, q_0-2$ and other mild conditions.  Then, for $i\in\{-q_0+1,-q_0+2,\dots, 0\}\cup\mathbb{N}$, $v_i$ is the amount of server effort or work measured in processing time units required by job $i$.
If $q_0>0$, the the work associated with the initial job with index $-q_0+1$ is being processed by the server at rate one beginning at time zero until either a new job arrives or time $v_{-q_0+1}$, which ever is smaller.
When a job arrives at a time when another job is currently being processed, it's processing time is compared with that remaining for job being processed.  If its processing time is strictly smaller, the job that was being processed is placed on hold and sent back into the queue, while the job that just arrived commences processing at rate one.
If the new arrival's processing time is greater or equal to that of the job being processed, the new arrival is placed in the queue and the processing of the current job at rate one continues.  Once the job indexed by $i$ receives a total of $v_i$ units of processing time, it departs the system.  If there are other jobs waiting in the queue at the time of a departure, the job waiting in queue with the smallest remaining processing time commences processing at rate one.  If there are multiple jobs with the smallest remaining processing time, the one among those with the smallest index commences processing at rate one.   When there are no jobs in system, no processing takes place and the system is said to be empty.  A job that arrives to an empty system commences processing at rate one upon its arrival.

In order to track the system state, it is necessary to know the remaining processing time of each job in system.  A standard state descriptor for the SRPT queue is a measure in $\bM$ that consists of unit atoms located at the remaining processing time of each job in system.  In particular, for $t\ge 0$, let
\begin{equation}\label{def:mvsd}
\cQ_t=\sum_{i=-q_0+1}^{E(t)} \delta_{v_i(t)}^+,
\end{equation}
where $\delta_x^+\in\bM$ is a unit atom at $x$ if $x\in(0,\infty)$ and is the zero measure otherwise and $v_i(t)$ is the remaining processing time of job $i$  at time $t$ for $i=-q_0+1,-q_0+2,\dots, E(t)$.  Then $\cQ_\cdot$ is the measure valued state descriptor process.  The queue length process $Q$ is given by $Q_t=\langle 1,\cQ_t\rangle$ for $t\geq 0$ and the workload process is given by $\langle \iota , \cQ_t\rangle$ for $t\geq 0$, where $\iota(x):=x$ for all $x\in\R_+$.

\subsection{Scaling Limit Theorems}\label{sec:ScalingLimits}
Fluid, or functional law of large numbers (FLLN), limits for the measure valued state descriptor \eqref{def:mvsd} were established in \cite{Down, Down_Sig}.  The works \cite{Atar, KrukSoko} develop FLLN limits for time-inhomogeneous settings and multiple job classes settings respectively.  The works \cite{Chen, Dong} use SRPT FLLN limits to study scheduling questions.  Heavy traffic functional central limit theorems (FCLT) for \eqref{def:mvsd} are established in \cite{Gromoll}.  To describe these results, we make note that under the standard heavy traffic conditions the FCLT scaled workload processes converge in distribution to a reflecting Brownian motion  (see \cite{Iglehart and Whitt}), which we denote as $\widetilde{W}_{\cdot}(\infty)$ here.
In \cite{Gromoll} for bounded support, it is shown that the FCLT scaling limit of \eqref{def:mvsd} is a point mass at the right edge $\widetilde{x}$ of the support of the processing time distribution with the total mass fluctuating according to $\widetilde{W}_{\cdot}(\infty)$ divided by $\widetilde{x}$, i.e, $\frac{\widetilde{W}_{\cdot}(\infty)}{\widetilde{x}}\delta_{\widetilde{x}}^+$.  In \cite{Gromoll} for unbounded support, it is shown that the FCLT scaling limit is the measure valued process that is identically equal to the zero measure.  In other words in the case of unbounded support, there is an order of magnitude difference between the FCLT scaled workload and queue length processes.

The works \cite{heavy tails, light tails, Puha} investigate this order of magnitude difference more deeply.  Using a distribution dependent scaling factor inspired by the left edge of the support of the fluid limit obtained in \cite{Down, Down_Sig}, the authors correct for the order of magnitude difference observed in \cite{Gromoll} to obtain nontrivial scaling limits for \eqref{def:mvsd} in the case of unbounded support.  Specifically, in \cite{heavy tails, light tails}, the authors consider a sequence of systems indexed by a sequence $\mathcal{R}=\{r\}$ of positive parameters that tend to infinity and satisfy the standard heavy traffic conditions as well as other mild conditions and such that all systems in the sequence have a common processing time distribution with CDF $F$ satisfying $F(x)<1$ for all $x\in\R_+$.  They define
$$
S(x)=\frac{1}{\int_x^\infty y dF(y)}\quad\hbox{for $x\in\R_+$ and}\quad S^{-1}(y)=\inf\{ x\geq 0 : S(x)>y\},\quad\hbox{for $y\in\R_+$.}
$$
Then, for each $r\in\mathcal{R}$, they set $c^r=S^{-1}(r)$ and define
\begin{equation}\label{eqn:ddsmvsd}
\widetilde{\cQ}_t^{r}=\frac{c^r}{r}\sum_{i=-q_0^r+1}^{E^r(r^2t)} \delta_{\frac{v_i^r(r^2t)}{c^r}}^+,\qquad t\ge 0.
\end{equation}
This is standard FCLT scaling, but with the addition of a boosting factor of $c^r$ to prevent the total mass from vanishing in the limit.  In addition, mass at $x$ is relocated to $x/c^r$ to prevent it from drifting out to infinity in the limit.  With this distribution depending scaling and other natural asymptotic conditions, scaling limits are obtained in \cite{heavy tails, light tails} for two separate cases described below.

In \cite{light tails}, the processing time CDF $F$ is assumed to have light tails.  Specifically, $F$ is assumed to satisfy that for all $c>1$
\begin{equation}\label{eqn:lt}
\lim_{x\to\infty} \frac{1-F(cx)}{1-F(x)}=0.
\end{equation}
Distributions such as the exponential and more generally the Weibull distribution satisfy \eqref{eqn:lt}.  In \cite{light tails} under condition \eqref{eqn:lt}, the authors show a sharp concentration of mass at one as $r\to\infty$ in \eqref{eqn:ddsmvsd} such that the scaling limit is given by $\widetilde{W}(\infty)\delta_1^+$.  We remark that the stationary behavior of this scaling limit as well as the one for bounded support is straightforward to recover from that of the process $\widetilde{W}_{\cdot}(\infty)$.  

\subsection{Scaling Limits for Heavy Tails}\label{sec:ht}
The work here applies to the scaling limit obtained in \cite{heavy tails} where the processing time CDF $F$ is assumed to have heavy tails.  Specifically, $F$ is assumed to satisfy that for some $p>1$ and for all $c>1$
\begin{equation}\label{eqn:ht}
\lim_{x\to\infty} \frac{1-F(cx)}{1-F(x)}=c^{-p-1}.
\end{equation}
Distributions such as the Pareto with suitable parameters satisfy \eqref{eqn:ht}.
In \cite{heavy tails} under the condition \eqref{eqn:ht}, the authors show that as $r\to\infty$ in \eqref{eqn:ddsmvsd} the limiting mass disperses in a manner governed by a certain time-space random field $(\widetilde{W}_t(a), a,t\ge 0)$.
More specifically, let
$\tilde \lambda=\lim_{r\to\infty}\lambda^r$ denote the limiting arrival rate, $\tilde\sigma_a=\lim_{r\to\infty}\sigma_a^r$ denote the limiting interarrival standard deviation and $\kappa=\lim_{r\to\infty}r(1-\lambda^r\mathbb{E}[v_1])$. Here we require $\kappa\in(0,\infty)$, and make note that $\tilde\lambda=1/\mathbb{E}[v_1]$.  Then set
\begin{equation}\label{def:SRPTmu}
\widetilde{\mu}(a)=\begin{cases}\kappa + \tilde\lambda a^{-p},& a\in(0,\infty),\\ \kappa,& a=\infty.\end{cases}
\end{equation}
Note that \eqref{def:SRPTmu} satisfies \eqref{ass:mu} since $p>1$ implies that $\int_0^1 \frac{1}{a^2\widetilde{\mu}(a)} da<\infty$ and $\kappa>0$
implies that $\int_1^\infty \frac{1}{a^2\widetilde{\mu}(a)} da<\infty$.
Also define $\tilde\sigma=\sqrt{\tilde\lambda((\tilde\sigma_a)^2+(\sigma_s)^2)}$ and, for $t\geq 0$, let
\begin{equation}\label{def:tildeX}
\widetilde{X}_t(a)=\begin{cases}0,&\hbox{for }a=0,\\
\tilde\sigma B(t)-\widetilde{\mu}(a)t, &\hbox{for }a\in(0,\infty].
\end{cases}
\end{equation}
Then, for $w\in\cI$, $a\in\overline\R_+$ and $t\geq 0$,
$\widetilde{W}_t(a)=\Psi[\widetilde{\chi}(a)](t)$, where $\widetilde{\chi}_s(a)=w(a)+ \widetilde{X}_s(a)$ for all $s\geq 0$.
For $t\geq 0$, let $\widetilde{\cZ}_t$ be given by \eqref{def:cZ0}--\eqref{def:Zt} with $W_t(\cdot)=\widetilde{W}_t(\cdot)$.  Then \cite[Theorem 3]{heavy tails} provides conditions under which the measure valued process $\widetilde{\cQ}_\cdot^{r}$ converges in distribution to $\widetilde{\cZ}_\cdot$ as $r\to\infty$ in $\D([0,\infty),\bM)$.  We refer to
$((\widetilde{W}_t(\cdot),\widetilde{W}_t(\infty),\widetilde{\cZ}_t),t\geq 0)$ as the {\it SRPT scaling limit for heavy tails}.

Perhaps it is worth noting that, due to the negative drift $-\kappa$, Proposition \ref{thm:hitszero} implies that the zero measure $\0$ is point recurrent for $\widetilde\cZ_\cdot$, demonstrating that in this case it reaches the state corresponding to no jobs in the system in finite time for all initial conditions $w \in \cI$. Intuitively, in this heavy traffic scaling limit with negative drift $-\kappa$, the SRPT strategy repeatedly  empties the system.

\subsection{The Stationary Behavior of Scaling Limits for Heavy Tails}\label{sec:SRPTthms}

To begin for $t\geq 0$ and for each $a\in\overline\R_+$, let
\begin{equation}\label{def:tildeM*}
\widetilde{M}_*(a):=
\sup_{t\geq 0}\widetilde{X}_t(a).
\end{equation}
Then let $\cMtildestar\in\bM$ be defined as in \eqref{def:cM*0}, \eqref{def:cM*} and \eqref{def:cM*TotMass} with $M_*(\cdot)=\widetilde{M}_*(\cdot)$. Specializing the result in Theorem \ref{thm:WStat} to SRPT scaling limits for heavy tails and using that $\tilde\mu(\infty)=\kappa>0$,
we obtain the following corollary.

\bigskip
\begin{cor} \label{cor:WStat}
For all $w\in\mathcal{I}$,
$(\widetilde{W}_t^w,\widetilde{W}_t^w(\infty),\widetilde{\cZ}_t^w)\convdist(\widetilde{M}_*, \widetilde{M}_*(\infty), \cMtildestar)$ in $\bC(\mathbb{R}_+)\times \mathbb{R}_+\times\bM$  as $t\to\infty$.  In particular, the distribution of
$(\widetilde{M}_*(\cdot), \widetilde{M}_*(\infty), \cMtildestar)$ is the unique stationary distribution of the SRPT scaling limit for heavy tails.
\end{cor}

Specializing the result in Corollary \ref{cor:MCov} to $(\widetilde{M}_*(\cdot),\widetilde{M}_*(\infty))$ gives the following covariance formula, where  $a_1/\infty=0$ by convention.
\bigskip
\begin{cor}\label{cor:tildeMCov}
For $0< a_1<a_2\leq \infty$,
$$
\hbox{Cov}(\widetilde{M}_*(a_1),\widetilde{M}_*(a_2)-\widetilde{M}_*(a_1))=
\frac{\tilde\sigma^4\tilde\lambda a_1^{2p}}{4(\kappa a_1^p+\tilde\lambda)^3}\left( 1- \frac{a_1^{p}}{a_2^p}\right).
$$
\end{cor}

As a consequence, we see that for $0<a_1<\infty$,
$$
\lim_{a_2\to a_1^{+}}\frac{\hbox{Cov}(\widetilde{M}_*(a_1),\widetilde{M}_*(a_2)-\widetilde{M}_*(a_1))}{a_2-a_1}= \frac{\tilde\sigma^4\tilde\lambda a_1^{2p-1}p}{4(\kappa a_1^p+\tilde\lambda)^3},
$$
and similarly for $a_1$ increasing to $0<a_2<\infty$.

Continuing to the stationary queue length $\widetilde{Z}_*$ of the SRPT scaling limit for heavy tails,
by Corollary \ref{cor:WStat}, we have
\begin{equation}\label{def:tildeZ*}
\widetilde{Z}_*=\langle 1,\cMtildestar\rangle=\int_0^\infty \frac{\widetilde{M}_*(x)}{x^2}dx.
\end{equation}
The distribution of $\widetilde{Z}_*$ seems to be complex, which is in
contrast to other single server queues such as first come first serve and even SPRT queues with processing time distributions
with bounded support or light tails.  In these cases, the heavy traffic queue length scaling limit process
is a constant multiple of the workload scaling limit.  For processing time distributions with heavy tails,
we find that, even in stationarity, the heavy traffic queue length scaling limit depends on the entire
maximum process $\widetilde{M}_*(\cdot)$.

Finally, we turn our attention to analytically calculate some moments of $\widetilde{Z}_*$.

\bigskip
\begin{cor}\label{cor:tildeZstar}
\begin{equation}
\label{eq:ExptildeZstar}
\E[\widetilde{Z}_*]=
\frac{\tilde\sigma^2}{2\kappa}\cdot\left(\frac{\kappa}{\tilde\lambda}\right)^{1/p}\cdot  \frac{\pi/p}{\sin(\pi/p)}.
\end{equation}
Moreover, if $p\in(1,2)\cup(2,\infty)$
\begin{equation}
\label{eq:VartildeZstar}
\text{Var}[\widetilde{Z}_*]
=
\frac{\tilde\sigma^4}{4\kappa^2}\cdot\left(\frac{\kappa}{\tilde\lambda}\right)^{2/p}\cdot
\frac{p^2+2p+2}{p^2(p+1)}\cdot\frac{\pi/p}{\sin(\pi/p)}\cdot\frac{p-2}{\cos(\pi/p)},
\end{equation}
whereas if $p=2$, $\text{Var}[\widetilde{Z}_*]= \frac{5\tilde\sigma^4}{12\kappa\tilde\lambda}$.
\end{cor}

The proof of Corollary \ref{cor:tildeZstar} given in Section \ref{s:eqQ2} relies on Corollary \ref{cor:tildeMCov} and some detailed calculations.

\bigskip
\begin{rem} \label{rem:tildeZstar}
The relationship between moments of stationary queue length and workload SRPT scaling limits for heavy tails depends intimately on the tail decay rate.
As \( p \to 1 \), both \( \mathbb{E}[\widetilde{Z}_*] \) and \( \text{Var}[\widetilde{Z}_*] \) tend to infinity. From \eqref{eq:ExptildeZstar}, \eqref{eq:VartildeZstar} and l'Hôpital's rule,
\[
\lim_{p \to 1} (p-1) \mathbb{E}[\widetilde{Z}_*] =\frac{\tilde\sigma^2}{2\tilde\lambda} \quad \text{and} \quad \lim_{p \to 1} (p-1) \text{Var}[\widetilde{Z}_*] = \frac{5\tilde\sigma^4}{8\tilde\lambda^2},
\]
which shows that the rate of divergence for both the expectation and variance is \( (p-1)^{-1} \) as \( p \to 1 \).
In addition,
\[\lim_{p \to \infty}\E[\widetilde{Z}_*]
=\frac{\tilde\sigma^2}{2\kappa}=\E[M_*(\infty)]  \quad \text{and} \quad \lim_{p \to \infty} \text{Var}[\widetilde{Z}_*]
=\frac{\tilde\sigma^4}{4\kappa^2}=\text{Var}[M_*(\infty)],
\]
which suggests that as the tail becomes less heavy the stationary scaling limit approaches the light tailed scaling limit in some sense.
\end{rem}

\bigskip
\begin{rem}[Little's Law]\label{rem:Lin}
In \cite{Lin}, the authors study the sequence of mean stationary response times for a sequence of SPRT queues with Poisson arrival processes as they approach heavy traffic.  To restate their Theorem 3 here in our notation, let $ G(x)=1-\tilde\lambda/S(x)$ for $x\in\R_+$.
For each $r\in\mathcal{R}$, let $\rho^r=\lambda^r\mathbb{E}[v_1]$, $\kappa^r=r(1-\rho^r)$ and $\mathbb{E}[T^r]$ denote the mean stationary response time for the $r$th system.  By \cite[Theorem 3]{Lin} and the fact that in the case of Poisson arrivals,
$\tilde\sigma_a=1/\tilde\lambda=\mathbb{E}[v_1]$ so that $\tilde\sigma^2=\tilde\lambda\mathbb{E}[v_1^2]$ or equivalently
$\mathbb{E}[v_1^2]=\mathbb{E}[v_1]\tilde\sigma^2$,
$$
\lim_{r\to\infty}(1-\rho^r) G^{-1}(\rho^r)\mathbb{E}[T^r]=\frac{\mathbb{E}[v_1]\tilde\sigma^2}{2}\cdot\frac{\pi/p}{\sin(\pi/p)}
=\kappa\left(\frac{\tilde \lambda}{\kappa}\right)^{1/p}\mathbb{E}[v_1]\mathbb{E}[\widetilde{Z}_*].
$$
The constant $\kappa(\tilde\lambda/\kappa)^{1/p}$ is accounted for by comparing the different scaling factors. Indeed, for each $r\in\mathcal{R}$, we have
$$
\frac{c^r}{r}=\frac{1}{\kappa^r}\cdot\frac{S^{-1}(r)}{S^{-1}(\tilde \lambda r/\kappa^r)}\cdot(1-\rho^r) G^{-1}(\rho^r).
$$
Due to \eqref{eqn:ht}, $S^{-1}$ is regularly varying with index $1/p$.  Then, by \cite[Theorem 1.3.1]{Bingham}, we have
$\lim_{r\to\infty} S^{-1}(r)/S^{-1}(\tilde\lambda r/\kappa^r)= \left(\kappa/\tilde\lambda\right)^{1/p}$.  Thus,
\begin{equation}\label{eq:Lin}
\lim_{r\to\infty}\frac{c^r}{r}\mathbb{E}[T^r]=\mathbb{E}[v_1]\mathbb{E}[\widetilde{Z}_*],
\end{equation}
which is an illustration of Little's Law.  When combine with Corollary \ref{cor:WStat},  \eqref{eq:Lin} suggests that an interchange of limits result might hold, a topic of future work.
\end{rem}

Finally, we turn our attention to convergence of higher moments and obtain the following corollary of Theorem \ref{thm:WStat2}.

\bigskip
\begin{cor}\label{cor:WStat2} Suppose that $\gamma\geq 1$ and $w\in\cI$. 
\begin{itemize}
\item[(i)] For $a\in\overline\R_+$ such that $\mathbb{E}\left[(w(a))^\gamma\right]<\infty$, 
$\lim_{t\to \infty}\E[(\widetilde W_t(a)^\gamma]=\E[(\widetilde M_*(a))^\gamma]$.
\item [(ii)] If $\E\left[\int_0^1\left(\frac{w(a)}{x^{2}}\right)^\gamma dx\right]<\infty$ and either $p\geq 2$ or $p\in (1,2)$ and
$\gamma<\frac{1}{2-p}$, then $\lim_{t\to \infty}\E[(\widetilde Z_t)^\gamma]=\E[(\widetilde Z_*)^\gamma]$.
\end{itemize}
\end{cor}

\begin{proof} The result follows from Theorem \ref{thm:WStat2} since $\tilde\mu(\infty)>0$ and, for $\gamma=1$, we have
\begin{eqnarray*}
\int_0^\infty\frac{1}{x^2\tilde\mu(x)}dx=\int_0^\infty\frac{1}{x^2(\kappa+\tilde\lambda x^{-p})}dx&\leq&
\int_0^1\frac{1}{\tilde\lambda x^{2-p}}dx
+
\int_1^\infty\frac{1}{\kappa x^2}dx<\infty,
\end{eqnarray*}
and, for $\gamma>1$, we have
\begin{eqnarray*}
\int_1^\infty\frac{1}{x^{\gamma}(\kappa+\tilde\lambda x^{-p})^\gamma}dx&\leq& \int_1^\infty\frac{1}{\kappa^\gamma x^{\gamma}}dx<\infty,\\
\int_0^1\frac{1}{x^{2\gamma}(\kappa+\tilde\lambda x^{-p})^\gamma}dx&\leq& \int_0^1\frac{1}{\tilde\lambda^\gamma x^{(2-p)\gamma}}dx<\infty,
\end{eqnarray*}
where finiteness in the second line above follows from $(2-p)\gamma<1$, which holds since either $p\geq 2$ or $p\in (1,2)$ and $\gamma<\frac{1}{2-p}$.
\end{proof}

%--------------------------
\section{Proofs}\label{sec:proofs}

\subsection{Proof of Theorems \ref{thm:WStat} and \ref{thm:WStat2}}\label{sec:WStat} 

First we prove Theorem  \ref{thm:WStat}.  For this,
we begin with an analysis of the maximum process $(M_*(\cdot),M_*(\infty))$ (defined in \eqref{def:M*})
and the measure $\cM_*$ (defined in \eqref{def:cM*0}, \eqref{def:cM*} and \eqref{def:cM*TotMass}).
We remind the reader that $M_*(0)=0$, $M_*(a)<\infty$ for each $a\in(0,\infty)$ due to the
negative drift of $X_\cdot(a)$, and $M_*(\infty)<\infty$ if and only if $\mu(\infty)>0$.
Recall the definition of $M_t(\cdot)$ from \eqref{def:Mt}.  For $a\in\overline\R_+$ and $t\geq 0$, $M_t(a)\leq M_*(a)$. This together with \eqref{eq:M*(a)cdf} gives that
for $t\geq 0$,
$$
\mathbb{E}\left[\int_0^\infty \frac{M_t(x)}{x^2}dx\right]\leq \mathbb{E}\left[\int_0^\infty \frac{M_*(x)}{x^2}dx\right]=\int_0^\infty \frac{\mathbb{E}\left[M_*(x)\right]}{x^2}dx=\int_0^\infty \frac{\sigma^2}{2x^2\mu(x)}dx.
$$
Hence, if \eqref{ass:mu} holds,  then, for each $t\geq 0$,
\begin{equation}
\int_0^\infty \frac{M_t(x)}{x^2}dx<\infty\qquad\hbox{and}\qquad Z_*=\int_0^\infty \frac{M_*(x)}{x^2}dx<\infty.\label{eq:cMfinite}
\end{equation}
When \eqref{ass:mu} holds, let $(\cM_t,t\geq 0)$ be the $\bM$ valued stochastic process that satisfies the following: for each $t\ge 0$,
\begin{eqnarray}
\langle 1_{\{0\}},\cM_t\rangle&=&0 \label{def:cMt0} \\
\langle 1_{(0,a]},\cM_t\rangle&=&\int_0^a\frac{M_t(x)}{x^2}dx+\frac{M_t(a)}{a},\qquad\hbox{for }a\in(0,\infty),\label{def:cMta}\\
\langle 1,\cM_t\rangle&=&\int_0^\infty\frac{M_t(x)}{x^2}dx.\label{def:cMtTotMass}
\end{eqnarray}

We will show that $\cM_t\in\bM$ for each $t\geq 0$ and $\cM_*\in\bM$ as part of Lemma \ref{lem:M*} below.  For this the following basic result is used.
\bigskip
\begin{prop}\label{prop:rc} 
Suppose that $f:\R_+\to\R_+$ is continuous, nondecreasing, $f(0)=0$ and $\int_0^1\frac{f(x)}{x^2}dx<\infty$.  Then $\lim_{x\to 0^+}f(x)/x=0$.
\end{prop}
\begin{proof} Fix $\varepsilon>0$. Let $\delta>0$ be such that $\int_0^\delta \frac{f(x)}{x^2}dx<\varepsilon$.  For all $a\in(0,\delta)$, we have
$$
\varepsilon+\frac{f(a)}{\delta}> \int_a^\delta \frac{f(x)}{x^2}dx+\frac{f(a)}{\delta}\geq f(a)\left(\frac{1}{a}-\frac{1}{\delta}\right)+\frac{f(a)}{\delta}=\frac{f(a)}{a}.
$$
Letting $a\to 0^+$ and using the fact that $\varepsilon>0$ was arbitrary, completes the proof.
\end{proof}

It will be convenient to let $\cTstar=[0,\infty)\cup\{*\}$, which is regarded as an index set.
\bigskip
\begin{lem}\label{lem:M*}
For each $t\in\cTstar$, the process $M_t(\cdot)$ is strictly increasing and continuous with $M_t(0)=0$ and $\lim_{a\to\infty}M_t(a)=M_t(\infty)$ almost surely.  In particular, $\lim_{t\to\infty}\Vert M_*(\cdot)-M_t(\cdot)\Vert_L=0$ for all $L\in(0,\infty)$ and, if $\mu(\infty)>0$, $\lim_{t\to\infty}\Vert M_*(\cdot)-M_t(\cdot)\Vert_\infty=0$.
Moreover, if \eqref{ass:mu} holds, then $\cM_t\in\bM$ for each $t\in\cTstar$ and $\cM_t\xrightarrow{w}\cM_*$ almost surely as $t\to\infty$.
\end{lem}

\begin{proof}
Let $t\in\cTstar$.
By definition \eqref{def:Mt} if $t\in[0,\infty)$ and definition \eqref{def:M*} if $t=*$ and the fact that $-\mu(\cdot)$ is strictly increasing and continuous on $(0,\infty)$, $M_t:(0,\infty)\to\mathbb{R}_+$ is nondecreasing and continuous on $(0,\infty)$.  Since $X_s(0)=0$ for all $s\in[0,\infty)$, then $M_t(0)=0$.
Since $-\lim_{a\to 0^+}\mu(a)=-\infty$, $\lim_{a\to 0^+}M_t(a)=0$. 
Then $M_t:\R_+\to\R_+$ is nondecreasing and continuous on $\R_+$.
Since $-\lim_{a\to \infty}\mu(a)=-\mu(\infty)\leq 0$, $\lim_{a\to \infty}M_t(a)=M_t(\infty)$.

Since $(M_t(\cdot), t\geq 0)$ is a monotonically increasing sequence of continuous functions that converge pointwise to the continuous function $M_*(\cdot)$ as $t\to\infty$, uniform convergence on compact intervals follows from Dini's theorem.  Uniform convergence on $\R_+$ follows when $\mu(\infty)>0$ since $M_*(\infty)<\infty$ in that case.

Finally, assume that \eqref{ass:mu} holds.  Then due to \eqref{eq:cMfinite}, the integrals in \eqref{def:cMta} and \eqref{def:cMtTotMass} are finite for each $t\in[0,\infty)$ as are the integrals in \eqref{def:cM*} and \eqref{def:cM*TotMass}.
For each $a\in\overline\R_+$ and $t\in\cTstar$, let $H_t(a)=\langle 1_{[0,a]},\cM_t\rangle$.  Then, due to  \eqref{eq:cMfinite} and Proposition \ref{prop:rc}, $H_t(\cdot)$ is right continuous for each $t\in\cTstar$.  Moreover, upon applying integration by parts to the integrals in \eqref{def:cM*} and \eqref{def:cMta}, we find that for all $t\in\cTstar$, $a\in\R_+$ and $\varepsilon>0$, $H_t(a+\varepsilon)-H_t(a)=\int_a^{a+\varepsilon}\frac{dM_t(x)}{x}dx\geq 0$.  Then $\cM_t\in\bM$ for each $t\in\cTstar$. By the definitions of $(\cM_t, t\geq 0)$ and $\cM_*$ and the monotone convergence theorem, $\lim_{t\to\infty}H_t(a)=H_*(a)$ for all $a\in\overline\R_+$.  The ``Moreover" statement follows from this.
\end{proof}

In preparation for proving Theorem \ref{thm:WStat}, we first consider $(W_t^0(\cdot),W_t^0(\infty))$,
the state of the RCBM at time $t$ starting from the empty state, i.e., $w(\cdot)=0(\cdot)$ and prove the following lemma.
Henceforth, $\eqdist$ denotes equal in distribution.
\bigskip
\begin{lem}\label{lem:EqualDist}
For each fixed $t\ge 0$, $(W_t^0(\cdot),W_t^0(\infty))\eqdist(M_t(\cdot),M_t(\infty))$
and if \eqref{ass:mu} holds, then $(W_t^0(\cdot),W_t^0(\infty),\cZ_t)\eqdist(M_t(\cdot),M_t(\infty),\cM_t)$.
In particular, Theorem \ref{thm:WStat} holds with $w(\cdot)=0(\cdot)$.
\end{lem}

\begin{proof}
Fix $t\ge 0$.  Define
$$
B_s^t=\begin{cases} B_t-B_{t-s},&\hbox{for }s\in[0,t],\\ B_s,&\hbox{for }s>t.\end{cases}
$$
Then $(B_s^t, s\geq 0)$ is a standard Brownian motion.  For each $a\in(0,\infty]$,
\begin{eqnarray*}
W_t^0(a)
&=&X_t(a)-\inf_{s\in[0,t]}X_s(a)=\sup_{s\in[0,t]}\left(X_t(a)-X_s(a)\right)=\sup_{s\in[0,t]}\left(X_t(a)-X_{t-s}(a)\right)\\
&=&\sup_{s\in[0,t]}\left(\sigma \left(B_t-B_{t-s}\right)-\mu(a)s\right)=\sup_{s\in[0,t]}\left(\sigma B^t_s-\mu(a)s\right).
\end{eqnarray*}
Since $\left(\sup_{s\in[0,t]}\left(\sigma B^t_s-\mu(a)s\right),a\in(0,\infty]\right)\eqdist\left(M_t(a),a\in(0,\infty]\right)$, it follows that $\left(W_t^0(a),a\in(0,\infty]\right)\eqdist\left(M_t(a),a\in(0,\infty]\right)$.  Since $W_t^0(0)=0=M_t(0)$ and $t\geq 0$ is arbitrary,
the first statement in Lemma \ref{lem:EqualDist} holds.
The result in Theorem \ref{thm:WStat} holds for $w(\cdot)=0(\cdot)$ due to Lemma \ref{lem:M*} and the first statement in Lemma \ref{lem:EqualDist}.
\end{proof}

\begin{proof}[Proof of Theorem \ref{thm:WStat}]
Fix $w\in\cI$.  For each $a\in\overline\R_+$, define $T_a^w:=\inf\{s\ge 0, \chi_s^w(a)=0\}$.  
Then $T_0^w=0$.  Also, for each $a\in(0,\infty]$, since $\mu(a)\geq 0$, $\chi_\cdot^w(a)$ is a Brownian motion with nonpositive drift and nonnegative initial value and $\mathbb{P}(T_a^w<\infty)=1$.

Fix $a\in\overline\R_+$. We begin by observing that $T_a^w$ is a coupling time of processes $W_\cdot^0(a)$ and $W^w_\cdot(a)$, i.e., it is a stopping time satisfying that for all $t\ge T_a^w$, $W_t^0(a)=W^w_t(a)$ almost surely.
To show this, by the definition of $T_a^w$, we have
\begin{eqnarray*}
0&=&\chi_{T_a^w}^w(a)=w(a)+X_{T_a^w}(a),\quad\hbox{so that }X_{T_a^w}(a)=-w(a).
\end{eqnarray*}
Thus, for all $t\ge T_a^w$, $\inf_{s\in[0,t]}X_s(a)\leq -w(a)$.
Hence, for any $t\ge T_a^w$,
\begin{eqnarray*}
W^w_t(a)
&=&w(a)+X_t(a)-\left(w(a)+\inf_{s\in[0,t]}X_s(a)\right)\wedge 0\\
&=&w(a)+X_t(a)-w(a)-\inf_{s\in[0,t]}X_s(a)\\
&=&X_t(a)-\inf_{s\in[0,t]}X_s(a)=X_t(a)-\left(\inf_{s\in[0,t]}X_s(a)\right)\wedge 0=W^0_t(a),
\end{eqnarray*}
as desired.

Next we make note that $T_\infty^w$ is a coupling time for the processes $((W_t^w(\cdot),W_t^w(\infty)),t\geq 0)$
and $((W_t^0(\cdot),W_t^0(\infty)),t\geq 0)$.  This follows because $T_a^w$ is a coupling time for the processes
$W_\cdot^0(a)$ and $W^w_\cdot(a)$  and $T_a^w\leq T_\infty^w$ for each $a\in\overline\R_+$ due to the fact that $\chi_t^w(a)\leq \chi_t^w(\infty)$ for all $t\geq 0$.  Thus, $(W_t^w(\cdot),W_t^w(\infty))=(W_t^0(\cdot),W_t^0(\infty))$ for all $t\geq T_\infty^w$.
For each $\varepsilon>0$, we have
$$
\mathbb{P}\left( \Vert W_t^w(\cdot) -W_t^0(\cdot) \Vert_\infty>\varepsilon\right)\leq \mathbb{P}\left(T_\infty^w>t\right).
$$
Thus, $\Vert W_t^w(\cdot) -W_t^0(\cdot) \Vert_\infty$ converges in probability to zero as $t\to\infty$.  This, together with Lemma \ref{lem:EqualDist} and the converging together lemma, completes the proof of Theorem \ref{thm:WStat} (i).

To verify (ii), assume that \eqref{ass:mu} holds. First we show that $\cZ_t^w\in\bM$ for each $t\ge 0$.
This follows similarly to the proof that $\cM_t\in\bM$ for each $t\geq 0$ once we show that $Z_t^w<\infty$ almost surely for each $t\geq 0$. Due to property \eqref{eq:Lip} of the Skorokhod map, for any $a\in\overline\R_+$ and $t\ge 0$, 
$
\left| W_t^w(a)-W_t^0(a)\right|\le 2\sup_{s\in[0,t]}\left| \chi_s(a)-X_s(a)\right|=2w(a), 
$
which implies that for any $a\in\overline\R_+$ and $t\ge 0$,
\begin{equation}\label{eq:WtwaBnd}
W_t^w(a)\leq 2w(a)+W_t^0(a).
\end{equation}
Thus, for all $t\geq 0$,
\begin{equation}\label{eq:ZtwBnd}
Z_t^w=\int_0^\infty\frac{W_t^w(x)}{x^2} dx\le 2\int_0^\infty\frac{w(x)}{x^2} dx+\int_0^\infty\frac{W_t^0(x)}{x^2} dx=2Z_0^w+Z_t^0.
\end{equation}
Due to \eqref{eq:Z0}, $Z_0^w<\infty$.  In addition for each $t\geq 0$, $Z_0^w$ and $Z_t^0$ are
independent random variables and, by Lemma \ref{lem:EqualDist}, $Z_t^0\eqdist\langle 1,\cM_t\rangle<\infty$ almost surely.
Combining these with \eqref{eq:ZtwBnd} and Proposition \ref{prop:rc} shows that $\cZ_t^w\in\bM$ for all $t\ge 0$.
Then, since $T_\infty^w$ is a coupling time for the processes $((W_t^w(\cdot),W_t^w(\infty),\cZ_t^w),t\geq 0)$
and $((W_t^0(\cdot),W_t^0(\infty),\cZ_t^0),t\geq 0)$, the balance of the result in Theorem \ref{thm:WStat} (ii)
follows similarly to the proof of (i).
\end{proof}

Next we prove Theorem \ref{thm:WStat2}.  For this, we follow a standard approach by establishing convergence in distribution and demonstrating uniform integrability to conclude convergence of the expected values.  In particular, we apply the following basic result.
\bigskip
\begin{prop}\label{prop:ui} Suppose that $(Y_t,t\geq 0)$ is an $\R_+$ valued stochastic process such that
$Y_t\convdist Y$  as $t\to\infty$ and $\mathbb{E}[Y]<\infty$.
If  $\lim_{m\to\infty}\sup_{t\geq 0}\mathbb{E}[Y_t 1_{(Y_t>m)}]=0$, then $\lim_{t\to\infty}\mathbb{E}[Y_t]=\mathbb{E}[Y]$.
\end{prop}

\begin{proof}  We use the inequality $x\wedge m\leq x\leq x\wedge m + x1_{\{x>m\}}$ for $x,m>0$.
Due to $Y_t\convdist Y$  as $t\to\infty$, $\lim_{t\to\infty} \mathbb{E}[Y_t\wedge m]=\mathbb{E}[Y\wedge m]$ for all $m>0$.
Then, for all $m>0$,
$$
\mathbb{E}[ Y\wedge m]\leq \liminf_{t\to\infty}\mathbb{E}[Y_t]
\leq \limsup_{t\to\infty}\mathbb{E}[Y_t]
\leq \mathbb{E}[ Y\wedge m]+\sup_{t\geq 0}\mathbb{E}[Y_t1_{\{Y_t>m\}}].
$$
Letting $m\to\infty$ completes the proof.
\end{proof}

Recall that the gamma function $\Gamma$ is given by $\Gamma[z]=\int_0^\infty t^{z-1}e^{-t} dt$, $z>1$.

\begin{proof}[Proof of Theorem \ref{thm:WStat2}]
Fix $w\in\cI$ and $\gamma\geq 1$.  Note that $\mathbb{E}[(w(a))^\gamma]<\infty$ for all $a\in[0,\infty)$ due to Fubini's theorem, \eqref{def:W0} in case $\gamma=1$ and \eqref{eq:highmoment} in case $\gamma>1$, and the monotonicity of $w(\cdot)$. 

First we prove (i).  Fix $a\in\overline\R_+$ such that $\mu(a)>0$.  We begin by observing that $\E[(M_*(a))^\gamma]=\frac{\Gamma[\gamma+1]\sigma^{2\gamma}}{2^\gamma\mu(a)^\gamma}<\infty$ since $M_*(a)$ has an exponential distribution with rate $2\mu(a)/\sigma^2$ (see \eqref{eq:M*(a)cdf}).   Next, by the continuous mapping theorem and Theorem \ref{thm:WStat}, $(W_t^w(a))^\gamma\convdist (M_*(a))^\gamma$ as $t\to\infty$.  Hence, it suffices to show that $((W_t^w(a))^\gamma,t\geq 0)$ is uniformly integrable.  Due to \eqref{eq:WtwaBnd}, the inequality
$(x+y)^\gamma\leq 2^{\gamma-1}(x^\gamma+y^\gamma)$ for all $x,y\geq 0$ (which follows from Jensen's inequality since $\gamma\geq 1$), $W_t^0(a)\eqdist M_t(a)$ for all $t\geq 0$ (see Theorem \ref{thm:WStat}) and $M_t(a)\leq M_*(a)$ for all $t\geq 0$, it follows that
$2^{2\gamma-1}(w(a))^\gamma+2^{\gamma-1}(M_*(a))^\gamma$ stochastically dominates $(W_t^w(a))^\gamma$ for each $t\geq 0$.  Then, since
$$
\mathbb{E}\left[ 2^{2\gamma-1}(w(a))^\gamma+2^{\gamma-1}(M_*(a))^\gamma\right]=
2^{2\gamma-1}\mathbb{E}\left[ (w(a))^\gamma\right]+2^{\gamma-1}\mathbb{E}\left[(M_*(a))^\gamma\right]<\infty,
$$
uniform integrability follows.

Now, we prove (ii).
By Theorem \ref{thm:WStat} and the continuous mapping theorem $(Z_t^w)^\gamma\convdist (Z_*)^\gamma$ as $t\to\infty$.
Hence, it suffices to show that $((Z_t^w)^\gamma,t\geq 0)$ is uniformly integrable.  Using \eqref{eq:WtwaBnd} and the inequality $(x+y)^\gamma\leq 2^{\gamma-1}(x^\gamma+y^\gamma)$ for all $x,y\geq 0$, for each $t\geq 0$ we have
\begin{eqnarray*}
(Z_t^w)^\gamma&=&\left(\int_0^\infty\frac{W_t^w(a)}{a^2}da\right)^\gamma\le \left(2\int_0^\infty\frac{w(a)}{a^2}da+\int_0^\infty\frac{W_t^0(a)}{a^2}da\right)^\gamma\\
&\leq& 2^{2\gamma-1}\left(\int_0^\infty\frac{w(a)}{a^2}da\right)^\gamma+2^{\gamma-1}\left(\int_0^\infty\frac{W_t^0(a)}{a^2}da\right)^\gamma= 2^{\gamma-1}\left(2(Z_0^w)^\gamma+(Z_t^0)^\gamma\right).
\end{eqnarray*}
Due to Lemma \ref{lem:M*}, $Z_t^0\eqdist\langle 1,\cM_t\rangle$ for all $t\geq 0$.
By the nondecreasing property of $M_t(\cdot)$, $\langle 1,\cM_t\rangle\leq Z_*$ for all $t\geq 0$.  Thus,
$2^{\gamma-1}\left(2(Z_0^w)^\gamma+(Z_*)^\gamma\right)$ stochastically dominates $(Z_t^w)^\gamma$ for each $t\geq 0$.  Hence, in order to show uniform integrability of $((Z_t^w)^\gamma,t\geq 0)$, it suffices to show that $\mathbb{E}[(Z_0^w)^\gamma]<\infty$ and $\mathbb{E}[(Z_*)^\gamma]<\infty$.  The former holds by assumption.  Thus, we must show that
$\mathbb{E}\left[ (Z_*)^\gamma\right]<\infty$.
If $\gamma=1$, then
\begin{eqnarray*}
\mathbb{E}[Z_*]&=&\int_0^\infty\frac{\mathbb{E}[M_*(a)]}{a^2}da=\frac{\sigma^2}{2}\int_0^\infty\frac{1}{\mu(a)a^2}da<\infty,
\end{eqnarray*}
where the final equality follows from \eqref{ass:mu}.
Otherwise, $\gamma>1$.  Then, using the inequality $(x+y)^\gamma\leq 2^{\gamma-1}(x^\gamma+y^\gamma)$ for all $x,y\geq 0$ followed by Holder's inequality with $p=\gamma$ and $q=\gamma/(\gamma-1)$, we have
\begin{eqnarray*}
(Z_*)^\gamma&=&\left(\int_0^1\frac{M_*(a)}{a^2}da+\int_1^\infty\frac{M_*(a)}{a^2}da\right)^\gamma\\
&\leq& 2^{\gamma-1}\left(\int_0^1\frac{M_*(a)}{a^2}da\right)^\gamma+2^{\gamma-1}\left(\int_1^\infty\frac{M_*(a)}{a^2}da\right)^\gamma\\
&=& 2^{\gamma-1}\left(\int_0^1 1\times \frac{M_*(a)}{a^2}da\right)^\gamma+2^{\gamma-1}\left(\int_1^\infty\frac{1}{a}\times \frac{M_*(a)}{a}da\right)^\gamma\\
&\leq& 2^{\gamma-1}\int_0^1\frac{\left(M_*(a)\right)^\gamma}{a^{2\gamma}}da+2^{\gamma-1}\left(\int_1^\infty\frac{1}{a^{\gamma/(\gamma-1)}}da\right)^{\gamma-1}\int_1^\infty \frac{(M_*(a))^\gamma}{a^\gamma}da\\
&=& 2^{\gamma-1}\int_0^1\frac{\left(M_*(a)\right)^\gamma}{a^{2\gamma}}da+\left(2(\gamma-1)\right)^{\gamma-1}\int_1^\infty \frac{(M_*(a))^\gamma}{a^\gamma}da<\infty.
\end{eqnarray*}
This together with Fubini's theorem and \eqref{eq:highmoment} gives
$$
\mathbb{E}\left[ (Z_*)^\gamma\right]
\leq
\frac{\Gamma[\gamma+1]\sigma^{2\gamma}}{2}\left(\int_0^1\frac{1}{(\mu(a))^\gamma a^{2\gamma}}da
+
(\gamma-1)^{\gamma-1}\int_1^\infty \frac{1}{(\mu(a))^\gamma a^\gamma} da \right)<\infty,
$$
which completes the proof.\end{proof}

\subsection{Proof of Theorem \ref{thm:ZStat}}\label{sec:distriproof}
Consider $n\in\N$ and $\vec a\in\cA_n$.
If $a_n = \infty$ and $\mu(\infty) = 0$, then both sides of the equation in
Theorem \ref{thm:ZStat} are equal to $0$. 
Thus, in the remainder of the proof that follows, it suffices to consider $n\in\N$ and $\vec a\in\cA_n$
such that $\mu(a_n) > 0$.

For this, we find it convenient to introduce some new notation. 
For $\nu\in\R$ and $t\ge 0$, let
$$
X_t^{\nu}:=\sigma B_t-\nu t,\qquad
M_t^{\nu}:=\sup_{s\in[0,t]} X_s^{\nu}
\qquad\hbox{and}\qquad
M_*^{\nu}:=\sup_{s\in[0,\infty)} X_s^{\nu}.
$$
Given $n\in\N$ and $\vec a\in\cA_n$, let $\vec \nu\in\R_+^n$ be such that $\nu_i=\mu(a_i)$ for $i=1,2,\dots,n$.
Then, for $i=1,2,\dots,n$, $M_t(a_i)=M_t^{\nu_i}$ for all $t\geq 0$ and $M_*(a_i)=M_*^{\nu_i}$, and with this
convention, for $\vec x\in\R_+^n$,
$$
\P\left( M_*(a_i)\le x_i\hbox{ for }i=1,2,\dots, n\right)
=
\P\left( M_*^{\nu_i}\le x_i\hbox{ for }i=1,2,\cdots, n\right).
$$
which are the probabilities that we wish to compute.  Define
\begin{equation}\label{def:pn}
p_n^{\vec{a}}(\vec{x}):=\P\left( M_*^{\nu_i}\le x_i\hbox{ for }i=1,2,\dots, n\right),\qquad\hbox{for }\vec x\in\R_+^n.
\end{equation}
In order to prove Theorem \ref{thm:ZStat}, it suffices to show that for all
$n\in\N$, $\vec a\in\mathcal{A}_n$, and $\vec{x}\in\cX_n^{\vec a}$
\begin{equation}\label{eq:M*}
p_n^{\vec{a}}(\vec{x})=\P\left( U_{\tau_{n-1}}^*\le x_1\right)-\exp\left(\frac{-2\nu_n}{\sigma^2}\right)\P\left( V_{\tau_{n-1}}^*\le x_1\right),
\end{equation}
where $U^*$ and $V^*$ are as in \eqref{def:U} and \eqref{def:V} with $\mu(a_i)=\nu_i$ for $i=1,2,\dots,n$.

We begin with a lemma which states that establishing \eqref{eq:M*} for all $n\in\N$, $\vec a\in\mathcal{A}_n$, and $\vec{x}\in\cX_n^{\vec a}$ suffices to specify the finite dimensional distributions of $(M_*(\cdot),M_*(\infty))$.
\bigskip
\begin{lem}\label{lem:tau}
Suppose that $n\in\N$, $\vec a\in\mathcal{A}_n$, and $\vec{x}\in\R_+^n$. If $\vec{x}\in\cX_n^{\vec a}$,
then
\begin{equation}\label{imp:tau}
p_n^{\vec{a}}(\vec{x})=\P\left(\sup_{s\in[\tau_{i-1},\tau_i)}X_s^{\nu_i}\le x_i\hbox{ for }i=1,2,\dots, n\right).
\end{equation}
Otherwise, there exists $i'\in\{1,2,\dots,n\}$ such that $p_n^{\vec{a}}(\vec{x})=\P\left(M_*^{\nu_i}\le x_i\hbox{ for }i\neq i'\right)$.
\end{lem}

\begin{proof} For $i=1,2,\dots, n$, let $\ell_i(s)=\nu_i s + x_i$ for $s\ge 0$.  We have
\begin{eqnarray}
\{ M_*^{\nu_i}\le x_i\hbox{ for }i=1,2,\dots, n\}
&=&\{X_s^{\nu_i}\le x_i\hbox{ for }s\ge 0\hbox{ and }i=1,2,\dots, n\}\nonumber\\
&=&\left\{\sigma B_s\le \min_{i=1,2,\dots,n}\ell_i(s)\hbox{ for all }s\ge 0\right\}.\label{eq:min1}
\end{eqnarray}
Thus, if $\vec{x}\in\mathcal{X}_n:=\{ \vec{x}\in\R_+^n: 0\leq x_1<x_2<\cdots<x_n<\infty\}$, then for all $i\in\{i,\dots,n-1\}$,
we have $\ell_i(0)<\ell_{i+1}(0)$, $\nu_i>\nu_{i+1}$ and
$\ell_i(\tau_i)=\ell_{i+1}(\tau_i)$, and so, 
\begin{equation}\label{eq:min}
\min\{\ell_i(s),\ell_{i+1}(s)\}=
\begin{cases}
\ell_i(s),& s\in[0,\tau_i),\\
\ell_{i+1}(s),& s\in[\tau_i,\infty).
\end{cases}
\end{equation}

Suppose that $\vec{x}\in\cX_n^{\vec a}$. Then \eqref{ass:tau} and \eqref{eq:min} together imply that for $i=1,2,\dots,n$
and $s\in[\tau_{i-1},\tau_i)$
\begin{equation}\label{eq:min2}
\min_{k=1,2,\dots,n}\ell_k(s)=\ell_i(s).
\end{equation}
This together with \eqref{eq:min1} implies that when $\vec{x}\in\cX_n^{\vec a}$,
\begin{eqnarray}
&&\{ M_*^{\nu_i}\le x_i\hbox{ for }i=1,2,\dots, n\}\nonumber\\
&&\qquad=\left\{\sigma B_s \le \ell_i(s) \hbox{ for }s\in[\tau_{i-1},\tau_i)\hbox{ and }i\in\{1,2,\dots,n\} \right\}\nonumber\\
&&\qquad=\left\{X_s^{\nu_i} \le  x_i  \hbox{ for }s\in[\tau_{i-1},\tau_i)\hbox{ and }i\in\{1,2,\dots,n\} \right\},\label{eq:tau1}
\end{eqnarray}
and \eqref{imp:tau} follows.

Next suppose that $\vec{x}\in\mathcal{X}_n\setminus \cX_n^{\vec a}$.  Then there exist $i\in\{1,2,\dots,n-1\}$ and $i'\in\{i+1,i+2,\dots,n\}$ such that $\tau_{i'}\leq \tau_i$, then,
by changing the value of $i'$ if needed, without loss of generality we may assume that $\tau_0<\cdots<\tau_{i'-2}<\tau_{i'-1}$ and  $\tau_{i'}\leq \tau_{i'-1}$ for some $i'\in\{2,3,\dots,n\}$.  By considering $n=i'$ in \eqref{eq:min2}, we have that $\min_{k=1,2,\dots,i'}\ell_k(s)=\min_{k=1,2,\dots,i'-1}\ell_k(s)\hbox{ for }s\in[0,\tau_{i'-1})$.  And, by \eqref{eq:min} with $i=i'$ we have
$\min\{\ell_{i'}(s),\ell_{i'+1}(s)\}=\ell_{i'+1}(s)$ for $s\in[\tau_{i'},\infty)\subseteq[\tau_{i'-1},\infty)$.  Combining these last two statements we find that for all $s\ge 0$
$$
\min_{k=1,2,\dots,i'+1}\ell_k(s)=\min\left(\min_{k=1,2,\dots,i'-1}\ell_k(s),\ell_{i'+1}(s)\right).
$$
Then, for all $s\ge 0$,
$\min_{k=1,2,\dots,n}\ell_k(s)=\min_{k\neq i'}\ell_k(s)$.
This together with \eqref{eq:min1} implies that
\begin{eqnarray*}
\{ M_*^{\nu_i}\le x_i\hbox{ for }i=1,2,\dots, n\}
&=&\left\{\sigma B_s\le \min_{i\neq i'}\ell_i(s)  \hbox{ for }s\ge 0\right\}\\
&=&\left\{X_s^{\nu_i}\le x_i \hbox{ for }s\ge 0\hbox{ and }i\neq i'\right\}\\
&=&\left\{M_*^{\nu_i}\le x_i \hbox{ for }i\neq i'\right\}.
\end{eqnarray*}

Finally, suppose that $\vec{x}\in\R_+^n\setminus\mathcal{X}_n$.  Then there exist $i'\in\{1,2,\dots,n-1\}$ and $j\in\{i'+1,\dots,n\}$ such that $x_j\le x_{i'}$.  Then, since  $\ell_j(0)\le\ell_{i'}(0)$ and $\nu_j<\nu_{i'}$, we have $\ell_j(s)\le\ell_{i'}(s)$ for all $s\ge 0$.  Thus,
$
\{ M_*^{\nu_i}\le x_i\hbox{ for }i=1,2,\dots, n\}=\left\{\sigma B_s\le \min_{i\ne i'}\ell_i(s)  \hbox{ for }s\ge 0\right\}.
$
\end{proof}

Before proceeding to the proof of \eqref{eq:M*}, we establish some notation and a basic fact.
For this, let $\phi$ denote the standard normal density function.
As in Proposition 8.1 in \cite{Harrison}, for $\nu\in\R_+$, $t\ge 0$, $x\in\R_+$ and $u\le x$,
\begin{eqnarray}
f_t^{\nu}(u,x)du
&:=&\mathbb{P}\left(X_t^{\nu}\in du,\ M_t^{\nu}\le x\right)\nonumber\\
&=&
\exp\left(\frac{-2\nu tu-\nu ^2t^2}{2\sigma^2t}\right)\left(\phi\left(\frac{u}{\sigma\sqrt{t}}\right)-\phi\left(\frac{u-2x}{\sigma\sqrt{t}}\right)\right)\frac{du}{\sigma\sqrt{t}}. \label{fact}
\end{eqnarray}
Let $\nu,\alpha\in\R$ and set $\zeta=\nu-2\alpha$.
We claim that for $t\ge 0$, $x\in\R_+$ and $u\le x$,
\begin{equation}\label{eq:basic}
f_t^{\zeta}(u,x)=\exp\left(\frac{2\alpha((\nu-\alpha)t+u)}{\sigma^2}\right)f_t^{\nu}(u,x).
\end{equation}
To see this, fix $t\ge 0$, $x\in\R$ and $u\le x$.
Note that
\begin{eqnarray*}
-2\nu tu-\nu ^2t^2+4\alpha((\nu-\alpha)t+u)t
&=&-2\zeta tu-\zeta^2t^2.
\end{eqnarray*}

\begin{proof}[Proof of \eqref{eq:M*}]
Now, we are ready to prove \eqref{eq:M*} by mathematical induction. Clearly, if $n=1$, \eqref{eq:M*} holds because, in this case, $\tau_{n-1}=0$ so that $\mathbb{P}\left( U_{\tau_{n-1}}^*\le x_1\right)=\mathbb{P}\left( V_{\tau_{n-1}}^*\le x_1\right)=1$.

Next, assume that \eqref{eq:M*} holds for $n=k-1$ with some $k\in\N$ and $k>1$.
Under this assumption, we show that  \eqref{eq:M*}  holds for $n=k$.
Then, by the principle of mathematical induction,  \eqref{eq:M*} holds for all $n\in\N$.

Fix $\vec{\nu}\in\mathcal{V}_k$ and $\vec{x}\in\mathcal{X}_k^{\vec{\nu}}$.  To begin, we rewrite \eqref{def:pn}
using Lemma \ref{lem:tau} as
\begin{equation}\label{def:pn2}
p_k^{\vec{v}}(\vec{x})=
\P\left(M_{\tau_1}^{\nu_i}\le x_1,\ \widehat M_*^{\nu_i}\le x_i\hbox{ for }i=2,\dots, k\right),
\end{equation}
where $\widehat M_*^{\nu_i}:=\sup_{s\in[\tau_1,\infty)}X_s^{\nu_i}$ for $i=2,\dots,k$.
In light of \eqref{def:pn2}, the proof proceeds by restarting the processes at time $\tau_1$ and using the Markov property.  

For this, fix $u\in(-\infty,x_1]$ and consider the event $A(u)=\{X_{\tau_1}^{\nu_1}=u\}$.
On $A(u)$, we have the following:
\begin{alignat*}{2}
X_{\tau_1}^{\nu_2}&=X_{\tau_1}^{\nu_1}+(\nu_1-\nu_2)\tau_1=(\nu_1-\nu_2)\tau_1+u=x_2-x_1+u,&\\
X_{\tau_1}^{\nu_i}
&=X_{\tau_1}^{\nu_1}+(\nu_1-\nu_i)\tau_1
=(\nu_1-\nu_i)\tau_1+u,&\quad\hbox{for }i=3,\dots,k.
\end{alignat*}
Hence, on the event $A(u)$, for $s\ge 0$ and $i=2,\dots,k$, $X_{\tau_1+s}^{\nu_i}\le x_i$ if and only if
\begin{equation}\label{eq:xiu}
X_{\tau_1+s}^{\nu_i}-X_{\tau_1}^{\nu_i}\le \tilde x_i-u,
\end{equation}
where $\tilde x_i= x_i -(\nu_1-\nu_i)\tau_1$ for $i=2,\dots, k$. 
In order to apply the induction hypothesis to the restarted processes
$(X_{\tau_1+s}^{\nu_i}-X_{\tau_1}^{\nu_i}, s\ge 0)$, $i=2,\dots,k$, we need to establish that
$$
0<\tilde x_2-u\quad\hbox{and}\quad \tilde x_i-u<\tilde x_{i+1}-u,\quad\hbox{for }i=2,\dots k-1,
$$
which, on noting that $\tilde x_2=x_1$, is equivalent to
$$
0<x_1-u\qquad\hbox{and}\qquad \tau_1<\frac{x_{i+1}-x_i}{\nu_i-\nu_{i+1}}=
\tau_i,\qquad \hbox{for }i=2,\dots,k-1.
$$
The above holds since $\vec{x}\in\mathcal{X}_{k}^{\vec{\nu}}$ and $u\in(-\infty,x_1)$.

Using the Markov property, \eqref{def:pn2} and \eqref{eq:xiu}, we have
\begin{equation}\label{eq:pnM}
p_k^{\vec{v}}(\vec{x})
=\int_{-\infty}^{x_1} f_{\tau_1}^{\nu_1}(u,x_1)\P\left( M_*^{\nu_i}\le \tilde x_i-u\hbox{ for }i=2,\dots,k\right)du.
\end{equation}
Let $\widetilde\tau=\tau_{k-1}-\tau_1$ and for $t\ge 0$, let $\widetilde U_t:=U_{\tau_1+t}-U_{\tau_1}$,  $\widetilde U_t^*:=\sup_{s\in[0,t]}\widetilde U_s$, $\widetilde V_t:=V_{\tau_1+t}-V_{\tau_1}$ 
and $\widetilde V_t^*:=\sup_{s\in[0,t]}\widetilde V_s$.
Then, using \eqref{eq:pnM}, $\tilde x_2=x_1$, and the induction hypothesis, it follows that
\begin{eqnarray}
p_k^{\vec{v}}(\vec{x})&=&\int_{-\infty}^{x_1} f_{\tau_1}^{\nu_1}(u,x_1)\mathbb{P}\left( \widetilde U_{\widetilde\tau}^*\le x_1-u\right)du\label{eq:pnih1}\\
&&\quad -\int_{-\infty}^{x_1} f_{\tau_1}^{\nu_1}(u,x_1)\exp\left(\frac{-2\nu_k(\tilde x_k-u)}{\sigma^2}\right)\mathbb{P}\left( \widetilde V_{\widetilde\tau}^*\le x_1-u\right)du.\label{eq:pnih2}
\end{eqnarray}
We see that
\begin{eqnarray}
\int_{-\infty}^{x_1} f_{\tau_1}^{\nu_1}(u,x_1)\mathbb{P}\left( \widetilde U_{\widetilde\tau}^*\le x_1-u\right)du
&=&
\int_{-\infty}^{x_1} f_{\tau_1}^{\nu_1}(u,x_1)\mathbb{P}\left( \widetilde U_{\widetilde\tau}^*+u\le x_1\right)du\nonumber\\
&=&\mathbb{P}\left(\widetilde U_{\tau_{k-1}}^*\le x_1\right).\label{eq:pnM1}
\end{eqnarray}
In addition, we have
\begin{eqnarray*}
&&\int_{-\infty}^{x_1}f_{\tau_1}^{\nu_1}(u,x_1)\exp\left(\frac{-2\nu_k(\tilde x_k-u)}{\sigma^2}\right)\mathbb{P}\left( \widetilde V_{\widetilde\tau}^*\le x_1-u\right)du\\
&&\qquad=\exp\left(\frac{-2\nu_kx_k}{\sigma^2}\right)\int_{-\infty}^{x_1}f_{\tau_1}^{\nu_1}(u,x_1)\exp\left(\frac{2\nu_k((\nu_1-\nu_k)\tau_1+u)}{\sigma^2}\right)\mathbb{P}\left( \widetilde V_{\widetilde\tau}^*\le x_1-u\right)du.
\end{eqnarray*}
Due to \eqref{eq:basic} with $\nu=\nu_1$, $\alpha=\nu_k$, $t=\tau_1$, and $x=x_1$, for each $u\in(-\infty,x_1]$, we have
$$
f_{\tau_1}^{\nu_1}(u,x_1)\exp\left( \frac{2\nu_k\left((\nu_1-\nu_k)\tau_1+u\right)}{\sigma^2}\right)=f_{\tau_1}^{\zeta_1}(u,x_1),
$$
where $\zeta_1=\nu_1-2\nu_2$.
Thus,
\begin{eqnarray}
&&\int_{-\infty}^{x_1}f_{\tau_1}^{\nu_1}(u,x_1)\exp\left(\frac{-2\nu_k(\tilde x_k-u)}{\sigma^2}\right)\mathbb{P}\left( \widetilde V_{\widetilde\tau}^*\le x_1-u\right)du\nonumber\\
&&\qquad=\exp\left(\frac{-2\nu_kx_k}{\sigma^2}\right)\int_{-\infty}^{x_1}f_{\tau_1}^{\zeta_1}(u,x_1)\mathbb{P}\left( \widetilde V_{\widetilde\tau}^*\le x_1-u\right)du\nonumber\\
&&\qquad=\exp\left(\frac{-2\nu_kx_k}{\sigma^2}\right)\mathbb{P}\left( \widetilde V_{\tau_{k-1}}^*\le x_1\right).\label{eq:pnM2}
\end{eqnarray}
Then by combining \eqref{eq:pnih1}, \eqref{eq:pnih2}, \eqref{eq:pnM1} and \eqref{eq:pnM2}, we see that \eqref{eq:M*} holds.
\end{proof}

%%%%%%%%%%%%%%%%%%%%%%%%%%%%%%%%%%%%%%%%%%%%%%%%%%%%

\subsection{Proof of Corollary \ref{cor:MCov}}\label{sec:cov}

We first recall some notation.
For $\nu\in\R$ and $t\ge 0$, let $X_t^{\nu}=\sigma B(t)-\nu t$ and $M_t^{\nu}=\sup_{s\in[0,t]} X_s^{\nu}$. 
Fix $0<\nu_2<\nu_1<\infty$ and $0<x_1<x_2<\infty$.  In order to prove Corollary \ref{cor:MCov}, it suffices to show that 
\begin{equation}\label{eq:cov1}
\hbox{Cov}(M^{\nu_1}_*,M^{\nu_2}_*)=
\frac{\sigma^4}{4\nu_1^2}\left(2-\frac{\nu_2}{\nu_1}\right).
\end{equation}
We have
$\hbox{Cov}(M_*^{\nu_1},M_*^{\nu_2})=\mathbb{E}\left[ M_*^{\nu_1}(M_*^{\nu_2}-M_*^{\nu_1})\right]+\mathbb{E}\left[(M_*^{\nu_1})^2\right]-\mathbb{E}\left[M_*^{\nu_1}\right]\mathbb{E}\left[M_*^{\nu_2}\right]$.
Using \eqref{eq:M*(a)cdf}, we find that
\begin{eqnarray*}
\mathbb{E}\left[(M_*^{\nu_1})^2\right]-\mathbb{E}\left[M_*^{\nu_1}\right]\mathbb{E}\left[M_*^{\nu_2}\right]=\frac{\sigma^4}{4\nu_1^2}\left(2-\frac{\nu_1}{\nu_2}\right).
\end{eqnarray*}
Thus, in order to show \eqref{eq:cov1}, it suffices to show that 
\begin{equation}\label{eq:cov2}
\mathbb{E}\left[ M_*^{\nu_1}(M_*^{\nu_2}-M_*^{\nu_1})\right]
=\frac{\sigma^4}{4\nu_1^2}\left(\frac{\nu_1}{\nu_2}-\frac{\nu_2}{\nu_1}\right)=\frac{\sigma^4}{4\nu_1^3\nu_2}\left(\nu_1^2-\nu_2^2\right).
\end{equation}
In order to verify \eqref{eq:cov2}, we will find the joint probability density function $g$ of
$M_*^{\nu_1}$ and $M_*^{\nu_2}-M_*^{\nu_1}$ and evaluated the integral $\int_0^\infty \int_0^\infty zx g(x,z)dx dz$.
For this, let $\delta_1=\nu_1-\nu_2$.  The next lemma provides the explicit formula for $g$ in terms of $\nu_1$ and $\delta_1$.
\bigskip
\begin{lem}\label{lem:pdfg}
The joint probability density function $g$ of $(M_*^{\nu_1},M_*^{\nu_2}-M_*^{\nu_1})$ is given by the following:
for $x,z>0$,
\begin{eqnarray}
&&g(x,z)=
\phi\left(\frac{x\delta_1+\nu_1 z}{\sigma\sqrt{\delta_1 z}}\right)\frac{2(\nu_1-\delta_1)\sqrt{\delta_1}(x+2z)}{\sigma^3\sqrt{z^3}}\label{eq:pdfg}\\
&&\quad+\frac{4(\nu_1-\delta_1)(2\delta_1-\nu_1)}{\sigma^4}\exp\left(\frac{-2\delta_1x-2(\nu_1-\delta_1) z}{\sigma^2}\right)\Phi\left(\frac{-x\delta_1-(\nu_1-2\delta_1) z}{\sigma\sqrt{\delta_1z}}\right),\nonumber
\end{eqnarray}
and is zero otherwise.
\end{lem}

Lemma \ref{lem:pdfg} is proved in Section \ref{sec:pdfs} below.
Given the result in Lemma \ref{lem:pdfg} in order to verify \eqref{eq:cov2}, it suffices to show
\begin{equation}\label{eq:cov3}
\int_0^\infty \int_0^\infty zx g(x,z)dx dz=\frac{\sigma^4}{4\nu_1^3\nu_2}\left(\nu_1^2-\nu_2^2\right).
\end{equation}
In the next section, we evaluate the integral above and verify \eqref{eq:cov3}.

\subsubsection{Verification of \eqref{eq:cov3}}\label{sec:exp}
We begin by calculating $\int_0^\infty xg(x,z)dx$.  We first observe a general fact: for $a>0$ and $b,c\in\R$,
with integration by parts,
\begin{eqnarray*}
&&\int_0^\infty xe^{-ax}\Phi(bx+c)dx=-\frac{1}{a}e^{-ax}x\Phi(bx+c)\Bigg|_0^\infty+\frac{1}{a}\int_0^\infty e^{-ax}d(x\Phi(bx+c))\\
&&\quad=\frac{1}{a}\int_0^\infty e^{-ax}\Phi(bx+c)dx+\frac{1}{a}\int_0^\infty e^{-ax}x\phi(bx+c)b\ dx\\
&&\quad=-\frac{1}{a^2}e^{-ax}\Phi(bx+c)\Bigg|_0^\infty+
\frac{1}{a^2}\int_0^\infty e^{-ax}\phi(bx+c)b dx+\frac{1}{a}\int_0^\infty e^{-ax}x\phi(bx+c)b\ dx\\
&&\quad=\frac{1}{a^2}\Phi(c)+
\int_0^\infty e^{-ax}\phi(bx+c)\left(\frac{1}{a^2}+\frac{x}{a}\right) b\ dx.
\end{eqnarray*}
With $a=2\delta_1/\sigma^2$ and $x>0$, we have
$\frac{1}{a^2}+\frac{x}{a}=\frac{\sigma^2(\sigma^2+2\delta_1x)}{4\delta_1^2}$.
Thus, for $z>0$,
\begin{eqnarray*}
&&\int_0^\infty x\exp{\left(\frac{-2\delta_1x}{\sigma^2}\right)}\Phi\left(\frac{-x\delta_1-(\nu_1-2\delta_1)z}{\sigma\sqrt{\delta_1z}}\right)dx
=\frac{\sigma^4}{4\delta_1^2}\Phi\left(\frac{-(\nu_1-2\delta_1)z}{\sigma\sqrt{\delta_1z}}\right)\\
&&\qquad-\int_0^\infty \exp{\left(\frac{-2\delta_1x}{\sigma^2}\right)}\phi\left(\frac{-x\delta_1-(\nu_1+2\delta_1)z}{\sigma\sqrt{\delta_1z}}\right)\frac{\sigma^2\left(\sigma^2+2\delta_1x\right)}{4\delta_1^2}\frac{\delta_1}{\sigma\sqrt{\delta_1z}}dx.
\end{eqnarray*}
And also applying 
$$
\exp\left(\frac{-2\delta_1x-2(\nu_1-\delta_1) z}{\sigma^2}\right)\phi\left(\frac{-x\delta_1-(\nu_1-2\delta_1) z}{\sigma\sqrt{\delta_1z}}\right)
=
\phi\left(\frac{x\delta_1+\nu_1z}{\sigma\sqrt{\delta_1z}}\right),\qquad x,z>0,
$$
for $z>0$ we have,
\begin{eqnarray*}
&&\int_0^\infty xg(x,z)dx
=
\frac{4(\nu_1-\delta_1)(2\delta_1-\nu_1)}{\sigma^4}\exp\left(\frac{-2(\nu_1-\delta_1)z}{\sigma^2}\right)\frac{\sigma^4}{4\delta_1^2}\Phi\left(\frac{-(\nu_1-2\delta_1)z}{\sigma\sqrt{\delta_1z}}\right)\nonumber\\
&&\qquad-\frac{4(\nu_1-\delta_1)(2\delta_1-\nu_1)}{\sigma^4}\int_0^\infty \phi\left(\frac{x\delta_1+\nu_1 z}{\sigma\sqrt{\delta_1z}}\right)\left(\frac{\sigma^2\left(\sigma^2+2\delta_1x\right)}{4\delta_1^2}\right)\frac{\delta_1}{\sigma\sqrt{\delta_1z}}dx\\
&&\qquad+\int_0^\infty \phi\left(\frac{x\delta_1+\nu_1 z}{\sigma\sqrt{\delta_1 z}}\right)\frac{2(\nu_1-\delta_1)\sqrt{\delta_1}(x+2z)x}{\sigma^3\sqrt{z^3}}dx.
\end{eqnarray*}
In the second term, we factor out $\sigma^2/(4\delta_1^2z)$ and distribute $2\delta_1-\nu_1$.
In the third term, we factor out $(\nu_1-\delta_1)/(\sigma^2\delta_1^2z)$.
Then we consolidating the two integrands to create one integral. With this,  we find that for $z>0$
\begin{eqnarray*}
&&\int_0^\infty xg(x,z)dx
=\frac{(\nu_1-\delta_1)(2\delta_1-\nu_1)}{\delta_1^2}\exp\left(\frac{-2(\nu_1-\delta_1)z}{\sigma^2}\right)\Phi\left(\frac{-(\nu_1-2\delta_1)z}{\sigma\sqrt{\delta_1z}}\right)\\
&&\qquad+\frac{(\nu_1-\delta_1)}{\sigma^2\delta_1^2z}\int_0^\infty\phi\left(\frac{x\delta_1+\nu_1z}{\sigma\sqrt{\delta_1z}}\right)\Big(2\delta_1^2x^2+2\nu_1\delta_1zx-(2\delta_1-\nu_1)\sigma^2z\Big)\frac{\delta_1}{\sigma\sqrt{\delta_1z}}dx.
\end{eqnarray*}
Using the change of variable $u=\frac{x\delta_1+\nu_1z}{\sigma\sqrt{\delta_1z}}$, we obtain
$$
2\delta_1^2x^2+2\nu_1\delta_1zx-(2\delta_1-\nu_1)\sigma^2z=2\sigma^2\delta_1 u^2z
-2\sigma\nu_1z\sqrt{\delta_1z}u-\sigma^2(2\delta_1-\nu_1)z,
$$
and so for $z>0$,
\begin{eqnarray*}
&&\int_0^\infty xg(x,z)dx=\frac{(\nu_1-\delta_1)(2\delta_1-\nu_1)}{\delta_1^2}\exp\left(\frac{-2(\nu_1-\delta_1)z}{\sigma^2}\right)\Phi\left(\frac{-(\nu_1-2\delta_1)z}{\sigma\sqrt{\delta_1z}}\right)\\
&&\qquad+\frac{(\nu_1-\delta_1)}{\sigma\delta_1^2}\int_{\frac{\nu_1\sqrt{z}}{\sigma\sqrt{\delta_1}}}^\infty\phi(u)\Big(2\sigma\delta_1u^2-2\nu_1\sqrt{\delta_1z}u-(2\delta_1-\nu_1)\sigma\Big)du.
\end{eqnarray*}
Using the identities
$$\int_a^\infty\phi(u)u^2du=a\phi(a)+\int_a^\infty\phi(u)du\qquad\hbox{and}\qquad \int_a^\infty\phi(u)udu=\phi(a),\qquad a>0,$$
we have
\begin{eqnarray*}
&&2\sigma\delta_1\int_{\frac{\nu_1\sqrt{z}}{\sigma\sqrt{\delta_1}}}^\infty\phi(u)u^2 du -2\nu_1\sqrt{\delta_1z}\int_{\frac{\nu_1\sqrt{z}}{\sigma\sqrt{\delta_1}}}^\infty \phi(u) u du-(2\delta_1-\nu_1)\sigma\int_{\frac{\nu_1\sqrt{z}}{\sigma\sqrt{\delta_1}}}^\infty\phi(u) du \\
&&\quad=\left(\frac{2\sigma\delta_1\nu_1\sqrt{z}}{\sigma\sqrt{\delta_1}}
-2\nu_1\sqrt{\delta_1z}\right)\phi\left(\frac{\nu_1\sqrt{z}}
{\sigma\sqrt{\delta_1}}\right)+\nu_1\sigma\int_{\frac{\nu_1\sqrt{z}}{\sigma\sqrt{\delta_1}}}^\infty\phi(u) du=\nu_1\sigma\int_{\frac{\nu_1\sqrt{z}}{\sigma\sqrt{\delta_1}}}^\infty\phi(u) du.
\end{eqnarray*}
Thus, for $z>0$,
\begin{eqnarray}
\int_0^\infty xg(x,z)dx&=&\frac{(\nu_1-\delta_1)(2\delta_1-\nu_1)}{\delta_1^2}\exp\left(\frac{-2(\nu_1-\delta_1)z}{\sigma^2}\right)\Phi\left(\frac{-(\nu_1-2\delta_1)z}{\sigma\sqrt{\delta_1z}}\right)\nonumber\\
&&\qquad+\frac{(\nu_1-\delta_1)\nu_1}{\delta_1^2}\int_{\frac{\nu_1\sqrt{z}}{\sigma\sqrt{\delta_1}}}^\infty\phi(u)du.\label{eq:int1}
\end{eqnarray}

Using \eqref{eq:int1} we calculate
\begin{eqnarray*}
&&\int_0^\infty z\int_0^\infty xg(x,z)dxdz\\
&&\quad=\frac{(\nu_1-\delta_1)(2\delta_1-\nu_1)}{\delta_1^2}\int_0^\infty z\exp\left(\frac{-2(\nu_1-\delta_1)z}{\sigma^2}\right)\Phi\left(\frac{-(\nu_1-2\delta_1)z}{\sigma\sqrt{\delta_1z}}\right)dz\\
&&\qquad+\frac{(\nu_1-\delta_1)\nu_1}{\delta_1^2}\int_0^\infty\int_{\frac{\nu_1\sqrt{z}}{\sigma\sqrt{\delta_1}}}^\infty z\phi(u)dudz.
\end{eqnarray*}
The integral in the second term on the right hand side of the above equation can be calculated as follows:
\begin{equation}
\label{AEq1}
\int_0^\infty\int_{\frac{\nu_1\sqrt{z}}{\sigma\sqrt{\delta_1}}}^\infty z\phi(u)dudz=\int_0^\infty\int_0^{\frac{\sigma^2\delta_1u^2}{\nu_1^2}}z\phi(u)dzdu=\frac{\sigma^4\delta_1^2}{2\nu_1^4}\int_0^\infty u^4\phi(u)du=\frac{3\sigma^4\delta_1^2}{4\nu_1^4}.
\end{equation}
For the integral in the first term, we observe a general fact that, for  $a>0$ and $b\in\R$, as a result of integration by parts with $u=\Phi(b\sqrt{u})$ and $v=-(1/a+u)e^{-au}/a$ and recognizing the resulting integrals in terms of the gamma function $\Gamma[\cdot]$, we have
\begin{eqnarray}
&&\int_0^\infty ue^{-au}\Phi(b\sqrt u)du=\frac{1}{2a^2}+\frac{b}{2a}\int_0^\infty \left(\frac{1}{a\sqrt{u}}+\sqrt{u}\right)e^{-au}\phi(b\sqrt{u})du\nonumber\\
&&\quad=\frac{1}{2a^2}+\frac{b}{2a^2\sqrt{2\pi}}\int_0^\infty u^{-1/2}e^{-(a+b^2/2)u}du+\frac{b}{2a\sqrt{2\pi}}\int_0^\infty u^{1/2}e^{-(a+b^2/2)u}du\nonumber\\
&&\quad=\frac{1}{2a^2}+\frac{b}{2a^2\sqrt{2\pi}\sqrt{a+b^2/2}}\Gamma[1/2]+\frac{b}{2a\sqrt{2\pi}\sqrt{(a+b^2/2)^3}}\Gamma[3/2]\nonumber\\
&&\quad=\frac{1}{2a^2}+\frac{b}{2a^2\sqrt{2a+b^2}}+\frac{b}{2a\sqrt{(2a+b^2)^3}}.\nonumber
\end{eqnarray}
Using this with $a=\frac{2(\nu_1-\delta_1)}{\sigma^2}$ and $b=-\frac{\nu_1-2\delta_1}{\sigma\sqrt{\delta_1}}$,
thus $2a+b^2=\frac{\nu_1^2}{\sigma^2\delta_1}$, together with \eqref{AEq1}, we finally obtain
\begin{eqnarray*}
&&\int_0^\infty z\int_0^\infty xg(x,z)dxdz\\
&&\quad=\frac{(\nu_1-\delta_1)(2\delta_1-\nu_1)}{\delta_1^2}\Bigg(\frac{\sigma^4}{8(\nu_1-\delta_1)^2}-\frac{\sigma^4(\nu_1-2\delta_1)}{8(\nu_1-\delta_1)^2\nu_1}-\frac{\sigma^4(\nu_1-2\delta_1)\delta_1}{4(\nu_1-\delta_1)\nu_1^3}\Bigg)\\
&&\qquad+\frac{(\nu_1-\delta_1)\nu_1}{\delta_1^2}\frac{3\sigma^4\delta_1^2}{4\nu_1^4}=\frac{\left(\nu_1^2-\nu_2^2\right)\sigma^4}{4\nu_1^3\nu_2},
\end{eqnarray*}
which is \eqref{eq:cov3} as desired.

\subsubsection{Joint Density Function Computation}\label{sec:pdfs}

Here we prove Lemma \ref{lem:pdfg}.  For this, we start with the joint cumulative distribution function of $M_*^{\nu_1}$ and $M_*^{\nu_2}$ as given in Theorem \ref{thm:ZStat_2d}.  We take the mixed partial derivatives to find the joint probability density function.  Then we perform a change of variables to find $g$, the joint probability density function of $M_*^{\nu_1}$ and $M_*^{\nu_2}-M_*^{\nu_1}$.

In the present notation, the result in Theorem \ref{thm:ZStat_2d} is
\begin{eqnarray*}
\P\left(M_*^{\nu_1}\le x_1,\ M_*^{\nu_2}\le x_2\right)=
\P\left(M_{\tau_1}^{\nu_1}\le x_1\right)-\exp\left(\frac{-2\nu_2x_2}{\sigma^2}\right)\P\left(M_{\tau_1}^{\zeta_1}\le x_1\right),
\end{eqnarray*}
where $\zeta_1=\nu_1-2\nu_2$ and $0\leq x_1<x_2<\infty$.
Using \eqref{eq:CDFMax} in this expression, we find
\begin{eqnarray*}
&&\P\left(M_*^{\nu_1}\le x_1,\ M_*^{\nu_2}\le x_2\right)
=\Phi\left(\frac{x_1+\nu_1\tau_1}{\sigma\sqrt{\tau_1}}\right)-\exp\left(\frac{-2\nu_1x_1}{\sigma^2}\right)\Phi\left(\frac{-x_1+\nu_1\tau_1}{\sigma\sqrt{\tau_1}}\right)\\
&&-\exp\left(\frac{-2\nu_2x_2}{\sigma^2}\right)\left(\Phi\left(\frac{x_1+\zeta_1\tau_1}{\sigma\sqrt{\tau_1}}\right)-\exp\left(\frac{-2\zeta_1x_1}{\sigma^2}\right)\Phi\left(\frac{-x_1+\zeta_1\tau_1}{\sigma\sqrt{\tau_1}}\right)\right),
\end{eqnarray*}
for $0\leq x_1<x_2<\infty$.

Next we calculate the joint density function of $(M_*^{\nu_1},M_*^{\nu_2})$, denoted by $h$, by taking partial derivatives with repect to $x_1$ and $x_2$, i.e., $h(x_1,x_2)=\frac{\partial}{\partial x_2}\frac{\partial}{\partial x_1}\P\left(M_*^{\nu_1}\le x_1,\ M_*^{\nu_2}\le x_2\right)$ for $0\leq x_1<x_2<\infty$. First, for $0\leq x_1<x_2<\infty$,
\begin{eqnarray*}
&&\frac{\partial}{\partial x_1}\P\left(M_*^{\nu_1}\le x_1,\ M_*^{\nu_2}\le x_2\right)\\
&&\quad=\phi\left(\frac{x_1+\nu_1\tau_1}{\sigma\sqrt{\tau_1}}\right)\frac{\partial}{\partial x_1}\frac{x_1+\nu_1\tau_1}{\sigma\sqrt{\tau_1}}
-\exp\left(\frac{-2\nu_1x_1}{\sigma^2}\right)\phi\left(\frac{-x_1+\nu_1\tau_1}{\sigma\sqrt{\tau_1}}\right)\frac{\partial}{\partial x_1}\frac{-x_1+\nu_1\tau_1}{\sigma\sqrt{\tau_1}}\\
&&\qquad +\frac{2\nu_1}{\sigma^2}\exp\left(\frac{-2\nu_1x_1}{\sigma^2}\right)\Phi\left(\frac{-x_1+\nu_1\tau_1}{\sigma\sqrt{\tau_1}}\right)\\
&&\qquad-\exp\left(\frac{-2\nu_2x_2}{\sigma^2}\right)\left(\phi\left(\frac{x_1+\zeta_1\tau_1}{\sigma\sqrt{\tau_1}}\right)\frac{\partial}{\partial x_1}\frac{x_1+\zeta_1\tau_1}{\sigma\sqrt{\tau_1}}\right.\\
&&\qquad\qquad\left.-\exp\left(\frac{-2\zeta_1x_1}{\sigma^2}\right)\phi\left(\frac{-x_1+\zeta_1\tau_1}{\sigma\sqrt{\tau_1}}\right)\frac{\partial}{\partial x_1}\frac{-x_1+\zeta_1\tau_1}{\sigma\sqrt{\tau_1}}\right)\\
&&\qquad-\exp\left(\frac{-2\nu_2x_2}{\sigma^2}\right)\frac{2\zeta_1}{\sigma^2}\exp\left(\frac{-2\zeta_1x_1}{\sigma^2}\right)\Phi\left(\frac{-x_1+\zeta_1\tau_1}{\sigma\sqrt{\tau_1}}\right).
\end{eqnarray*}
For any $\nu\in\R$ and $x\in\R$, we have
\begin{eqnarray*}
\exp\left(\frac{-2\nu x}{\sigma^2}\right)\phi\left(\frac{-x+\nu\tau_1}{\sigma\sqrt{\tau_1}}\right)&=&\exp\left(\frac{-2\nu x}{\sigma^2}\right)\phi\left(\frac{x-\nu\tau_1}{\sigma\sqrt{\tau_1}}\right)=\phi\left(\frac{x+\nu\tau_1}{\sigma\sqrt{\tau_1}}\right).
\end{eqnarray*}
Using the above and $x_1+\zeta_1\tau_1=x_2-\nu_2\tau_1$, we find that for $0\leq x_1<x_2<\infty$,
\begin{eqnarray*}
&&\frac{\partial}{\partial x_1}\P\left(M_*^{\nu_1}\le x_1,\ M_*^{\nu_2}\le x_2\right)\\
&&\quad=\phi\left(\frac{x_1+\nu_1\tau_1}{\sigma\sqrt{\tau_1}}\right)\left(\frac{\partial}{\partial x_1}\frac{x_1+\nu_1\tau_1}{\sigma\sqrt{\tau_1}}
-\frac{\partial}{\partial x_1}\frac{-x_1+\nu_1\tau_1}{\sigma\sqrt{\tau_1}}\right)\\
&&\qquad +\frac{2\nu_1}{\sigma^2}\exp\left(\frac{-2\nu_1x_1}{\sigma^2}\right)\Phi\left(\frac{-x_1+\nu_1\tau_1}{\sigma\sqrt{\tau_1}}\right)\\
&&\qquad-\exp\left(\frac{-2\nu_2x_2}{\sigma^2}\right)\phi\left(\frac{x_1+\zeta_1\tau_1}{\sigma\sqrt{\tau_1}}\right)\left(\frac{\partial}{\partial x_1}\frac{x_1+\zeta_1\tau_1}{\sigma\sqrt{\tau_1}}-\frac{\partial}{\partial x_1}\frac{-x_1+\zeta_1\tau_1}{\sigma\sqrt{\tau_1}}\right)\\
&&\qquad-\exp\left(\frac{-2\nu_2x_2}{\sigma^2}\right)\frac{2\zeta_1}{\sigma^2}\exp\left(\frac{-2\zeta_1x_1}{\sigma^2}\right)\Phi\left(\frac{-x_1+\zeta_1\tau_1}{\sigma\sqrt{\tau_1}}\right)\\
&&\quad=\phi\left(\frac{x_1+\nu_1\tau_1}{\sigma\sqrt{\tau_1}}\right)\frac{\partial}{\partial x_1}\frac{2x_1}{\sigma\sqrt{\tau_1}}+\frac{2\nu_1}{\sigma^2}\exp\left(\frac{-2\nu_1x_1}{\sigma^2}\right)\Phi\left(\frac{-x_1+\nu_1\tau_1}{\sigma\sqrt{\tau_1}}\right)\\
&&\qquad-\exp\left(\frac{-2\nu_2x_2}{\sigma^2}\right)\phi\left(\frac{x_2-\nu_2\tau_1}{\sigma\sqrt{\tau_1}}\right)\frac{\partial}{\partial x_1}\frac{2x_1}{\sigma\sqrt{\tau_1}}\\
&&\qquad-\exp\left(\frac{-2\nu_2x_2}{\sigma^2}\right)\frac{2\zeta_1}{\sigma^2}\exp\left(\frac{-2\zeta_1x_1}{\sigma^2}\right)\Phi\left(\frac{-x_1+\zeta_1\tau_1}{\sigma\sqrt{\tau_1}}\right)\\
&&\quad=\phi\left(\frac{x_1+\nu_1\tau_1}{\sigma\sqrt{\tau_1}}\right)\frac{\partial}{\partial x_1}\frac{2x_1}{\sigma\sqrt{\tau_1}}+\frac{2\nu_1}{\sigma^2}\exp\left(\frac{-2\nu_1x_1}{\sigma^2}\right)\Phi\left(\frac{-x_1+\nu_1\tau_1}{\sigma\sqrt{\tau_1}}\right)\\
&&\qquad-\phi\left(\frac{x_2+\nu_2\tau_1}{\sigma\sqrt{\tau_1}}\right)\frac{\partial}{\partial x_1}\frac{2x_1}{\sigma\sqrt{\tau_1}}\\
&&\qquad-\exp\left(\frac{-2\nu_2x_2}{\sigma^2}\right)\frac{2\zeta_1}{\sigma^2}\exp\left(\frac{-2\zeta_1x_1}{\sigma^2}\right)\Phi\left(\frac{-x_1+\zeta_1\tau_1}{\sigma\sqrt{\tau_1}}\right).
\end{eqnarray*}
Finally, using $x_2+\nu_2\tau_1=x_1+\nu_1\tau_1$ we obtain that for $0\leq x_1<x_2<\infty$
\begin{eqnarray*}
&&\frac{\partial}{\partial x_1}\P\left(M_*^{\nu_1}\le x_1,\ M_*^{\nu_2}\le x_2\right)=\frac{2\nu_1}{\sigma^2}\exp\left(\frac{-2\nu_1x_1}{\sigma^2}\right)\Phi\left(\frac{-x_1+\nu_1\tau_1}{\sigma\sqrt{\tau_1}}\right)\\
&&\qquad-\exp\left(\frac{-2\nu_2x_2}{\sigma^2}\right)\frac{2\zeta_1}{\sigma^2}\exp\left(\frac{-2\zeta_1x_1}{\sigma^2}\right)\Phi\left(\frac{-x_1+\zeta_1\tau_1}{\sigma\sqrt{\tau_1}}\right).
\end{eqnarray*}

Next, by taking the second partial derivative with respect to $x_2$, we obtain the joint density function $h$.  For $0\leq x_1<x_2<\infty$,
\begin{eqnarray}
&&h(x_1,x_2)=\frac{\partial}{\partial x_2}\frac{\partial}{\partial x_1}\P\left(M_*^{\nu_1}\le x_1,\ M_*^{\nu_2}\le x_2\right)\notag\\
&&\quad=\frac{2\nu_1}{\sigma^2}\exp\left(\frac{-2\nu_1x_1}{\sigma^2}\right)\phi\left(\frac{-x_1+\nu_1\tau_1}{\sigma\sqrt{\tau_1}}\right)\frac{\partial}{\partial x_2}\frac{-x_1+\nu_1\tau_1}{\sigma\sqrt{\tau_1}}\notag\\
&&\qquad-\exp\left(\frac{-2\nu_2x_2}{\sigma^2}\right)\frac{2\zeta_1}{\sigma^2}\exp\left(\frac{-2\zeta_1x_1}{\sigma^2}\right)\phi\left(\frac{-x_1+\zeta_1\tau_1}{\sigma\sqrt{\tau_1}}\right)\frac{\partial}{\partial x_2}\frac{-x_1+\zeta_1\tau_1}{\sigma\sqrt{\tau_1}}\notag\\
&&\qquad+\frac{2\nu_2}{\sigma^2}\exp\left(\frac{-2\nu_2x_2}{\sigma^2}\right)\frac{2\zeta_1}{\sigma^2}\exp\left(\frac{-2\zeta_1x_1}{\sigma^2}\right)\Phi\left(\frac{-x_1+\zeta_1\tau_1}{\sigma\sqrt{\tau_1}}\right)\notag\\
&&\quad=\phi\left(\frac{x_1+\nu_1\tau_1}{\sigma\sqrt{\tau_1}}\right)\left(\frac{2}{\sigma^2}\frac{\partial}{\partial x_2}\left[\frac{\nu_1(-x_1+\nu_1\tau_1)}{\sigma\sqrt{\tau_1}}-\frac{\zeta_1(-x_1+\zeta_1\tau_1)}{\sigma\sqrt{\tau_1}}\right]\right)\notag\\
&&\qquad+\frac{2\nu_2}{\sigma^2}\exp\left(\frac{-2\nu_2x_2}{\sigma^2}\right)\frac{2\zeta_1}{\sigma^2}\exp\left(\frac{-2\zeta_1x_1}{\sigma^2}\right)\Phi\left(\frac{-x_1+\zeta_1\tau_1}{\sigma\sqrt{\tau_1}}\right)\notag\\
&&\quad=\phi\left(\frac{x_1+\nu_1\tau_1}{\sigma\sqrt{\tau_1}}\right)\frac{4\nu_2}{\sigma^2}\frac{\partial}{\partial x_2}\left[\frac{2x_2-3x_1}{\sigma\sqrt{\tau_1}}
\right]\notag\\
&&\qquad+\frac{2\nu_2}{\sigma^2}\exp\left(\frac{-2\nu_2x_2}{\sigma^2}\right)\frac{2\zeta_1}{\sigma^2}\exp\left(\frac{-2\zeta_1x_1}{\sigma^2}\right)\Phi\left(\frac{-x_1+\zeta_1\tau_1}{\sigma\sqrt{\tau_1}}\right)\notag\\
&&\quad=\phi\left(\frac{x_1+\nu_1\tau_1}{\sigma\sqrt{\tau_1}}\right)\frac{2\nu_2}{\sigma^2}\left(\frac{\sqrt{\nu_1-\nu_2}(2x_2- x_1)}{\sigma\sqrt{(x_2-x_1)^3}}
\right)\label{h}\\
&&\qquad+\frac{2\nu_2}{\sigma^2}\exp\left(\frac{-2\nu_2x_2}{\sigma^2}\right)\frac{2\zeta_1}{\sigma^2}\exp\left(\frac{-2\zeta_1x_1}{\sigma^2}\right)\Phi\left(\frac{-x_1+\zeta_1\tau_1}{\sigma\sqrt{\tau_1}}\right). \notag
\end{eqnarray}

Finally, to determine $g$, we apply the change of variables $x=x_1$, $z=x_2-x_1$ and $\delta_1=\nu_1-\nu_2$ to \eqref{h}.  Since the relevant Jacobian determinant is one, we have $dx_1dx_2=dx dz$ and then $g(x,z)=h(x,x+z)$, which gives \eqref{eq:pdfg}.

\subsection{Proof of Corllary  \ref{cor:tildeZstar}}
\label{s:eqQ2}
We first state the following identity. Recall $p>1$. With the change of variable $t=\frac{1}{s^p+1}$,
we have  $s=t^{-1/p}(1-t)^{1/p}$ and $ds=-p^{-1}t^{-1-1/p}(1-t)^{-1+1/p}dt$, we have
\begin{eqnarray*}
\int_0^\infty\frac{s^{p-2}}{s^p+1}ds%&=&\frac1p\int_0^1  t^{-1+2/p}(1-t)^{1-2/p}t^1 t^{-1-1/p}(1-t)^{-1+1/p}dt\\
&=&\frac1p\int_0^1t^{\frac1p-1}(1-t)^{-\frac1p}ds=\frac1p\B\left(\frac1p,1-\frac1p\right)=\frac{\pi/p}{\sin(\pi/p)},
\end{eqnarray*}
where $\B$ is the Beta function, i.e., $\B(x,y)=\int_0^1 t^{x-1} (1-t)^{y-1} dt$ for $x,y>0$.
 Then, with \eqref{eq:M*(a)cdf}, $\mu(a)=\tilde\mu(a)=\kappa+\tilde\lambda a^{-p}$ for $a>0$ and the change of variable $a=(\tilde\lambda/\kappa)^{1/p}s$, we have
\begin{eqnarray*}
\E[\widetilde Z_*]&=&\int_0^\infty \frac{\E[\widetilde M_*(a)]}{a^2}da=\frac{\tilde\sigma^2}{2}\int_0^\infty\frac{1}{\tilde\mu(a)a^2}da=\frac{\tilde\sigma^2}{2}\int_0^\infty\frac{a^{p-2}}{\kappa a^p+\tilde\lambda}da\\
&=&\frac{\tilde\sigma^2}{2\kappa}\left(\frac{\kappa}{\tilde\lambda}\right)^{1/p}\int_0^\infty\frac{s^{p-2}}{s^p+1}ds=\frac{\tilde\sigma^2}{2\kappa}\left(\frac{\kappa}{\tilde\lambda}\right)^{1/p} \frac{\pi/p}{\sin(\pi/p)},
\end{eqnarray*}
which proves equation \eqref{eq:ExptildeZstar}.

To prove \eqref{eq:VartildeZstar}, we observe that
\begin{eqnarray*}
\E[\widetilde Z_*^2]&=&\int_0^\infty\int_0^\infty\frac{\E[\widetilde M_*(a)\widetilde M_*(b)]}{a^2b^2}dadb\\
\left(\E[\widetilde Z_*]\right)^2&=&\int_0^\infty\int_0^\infty\frac{\E[\widetilde M_*(a)]\E[\widetilde M_*(b)]}{a^2b^2}dadb.
\end{eqnarray*}
Thus,
\begin{eqnarray*}
\text{Var}(\widetilde Z_*)&=&\int_0^\infty\int_0^\infty\frac{\text{Cov}(\widetilde M_*(a),\widetilde M_*(b))}{a^2b^2}dadb
=2\int_0^\infty\int_a^\infty\frac{\text{Cov}(\widetilde M_*(a),\widetilde M_*(b))}{a^2b^2}dbda.
\end{eqnarray*}
By Corollary \ref{cor:MCov}, we have
\begin{eqnarray*}
\text{Var}(\widetilde Z_*)
&=&\frac{\tilde\sigma^4}{2}\int_0^\infty\int_a^\infty\frac{2\tilde\mu(a)-\tilde\mu(b)}{\tilde\mu(a)^3a^2b^2}dbda.
\end{eqnarray*}
Using $\tilde\mu(a)=\kappa+\tilde\lambda a^{-p}$ for $a>0$, we have
\begin{eqnarray*}
\text{Var}(\widetilde Z_*)
&=&\frac{\tilde\sigma^4}{2}\int_0^\infty\frac{1}{\tilde\mu(a)^3a^2}\int_a^\infty\left(\frac{2\tilde\mu(a)-\kappa}{b^2}-\frac{\tilde\lambda}{b^{p+2}}\right)dbda\\
&=&\frac{\tilde\sigma^4}{2}\int_0^\infty\frac{1}{\tilde\mu(a)^3a^2}\left(\frac{2\tilde\mu(a)-\kappa}{a}-\frac{\tilde\lambda}{(p+1)a^{p+1}}\right)da\\
&=&\frac{\tilde\sigma^4}{2}\int_0^\infty\frac{2\tilde\mu(a)-\kappa-\frac{\tilde\lambda}{(p+1)}a^{-p}}{\tilde\mu(a)^3a^3}da
=\frac{\tilde\sigma^4}{2}\int_0^\infty\frac{\kappa+2\tilde\lambda a^{-p}+\frac{\tilde\lambda}{p+1} a^{-p}}{(\kappa+\tilde\lambda a^{-p})^3a^3}da\\
&=&\frac{\tilde\sigma^4}{2}\int_0^\infty\frac{\kappa+\frac{\tilde\lambda(2p+3)}{p+1} a^{-p}}{(\kappa a^p+\tilde\lambda)^3a^{3-3p}}da
=\frac{\tilde\sigma^4}{2}\int_0^\infty\frac{\kappa a^{3p-3}+\frac{\tilde\lambda(2p+3)}{p+1} a^{2p-3}}{(\kappa+\tilde\lambda a^{-p})^3}da.
\end{eqnarray*}
Letting $c=\frac{2p+3}{p+1}$ and with the change of variable $a=(\tilde\lambda/\kappa)^{\frac1p}s$, we have
$$\int_0^\infty\frac{\kappa a^{3p-3}+\tilde\lambda ca^{2p-3}}{(\kappa a^p+\tilde\lambda)^3}da=\frac{1}{\tilde\lambda^3}\int_0^\infty\frac{As^{3p-3}+Bs^{2p-3}}{(s^p+1)^3}ds$$
where $A=\kappa (\tilde\lambda/\kappa)^{\frac{3p-2}{p}}$ and $B=\tilde\lambda c (\tilde\lambda/\kappa)^{\frac{2p-2}{p}}$.
Similarly, we use the change of variable $t=\frac{1}{s^p+1}$ and so
$s=t^{-1/p}(1-t)^{1/p}$ and $ds=-p^{-1}t^{-1-1/p}(1-t)^{-1+1/p}dt$,
\begin{eqnarray*}
\frac{1}{\tilde\lambda^3}\int_0^\infty\frac{As^{3p-3}+Bs^{2p-3}}{(s^p+1)^3}ds
&=&\frac{1}{p\tilde\lambda^3}\int_0^1\left(At^{-1+\frac2p}(1-t)^{2-\frac2p}+Bt^{\frac2p}(1-t)^{1-\frac2p}\right)dt\\
&=&\frac{A}{p\tilde\lambda^3}\B\left(\frac2p,3-\frac2p\right)+\frac{B}{p\tilde\lambda^3}\B\left(1+\frac2p,2-\frac2p\right).
\end{eqnarray*}

Combining with \eqref{eq:ExptildeZstar} and some identities of the beta function, we finally obtain 
\begin{eqnarray*}
\text{Var}(\widetilde Z_*)
&=&\frac{\tilde\sigma^4}{4\kappa^2}\left(\frac{\kappa}{\tilde\lambda}\right)^{2/p}\frac{2}{p}\left(\B\left(\frac2p,3-\frac2p\right)+\frac{2p+3}{p+1}\B\left(1+\frac2p,2-\frac2p\right)\right)\\
&=&\frac{\tilde\sigma^4}{4\kappa^2}\left(\frac{\kappa}{\tilde\lambda}\right)^{2/p}\frac{2}{p}\left(\frac{p-1}{p}\B\left(\frac2p,2-\frac2p\right)+\frac{2p+3}{p(p+1)}\B\left(\frac2p,2-\frac2p\right)\right).
\end{eqnarray*}
If $p=2$, then $\B\left(1,1\right)=1$ and $\text{Var}(\widetilde Z_*)=\frac{\tilde\sigma^4}{4\kappa^2}\left(\frac{\kappa}{\tilde\lambda}\right)\frac{5}{3}=\frac{5\tilde\sigma^4}{12\kappa\tilde\lambda}$, as desired.  Otherwise, $p\neq 2$ and we obtain \eqref{eq:VartildeZstar} in this case as follows:
\begin{eqnarray*}
\text{Var}(\widetilde Z_*)
&=&\frac{\tilde\sigma^4}{4\kappa^2}\left(\frac{\kappa}{\tilde\lambda}\right)^{2/p}\frac{2}{p}\left(\frac{p-1}{p}+\frac{2p+3}{p(p+1)}\right)\left(\frac{p-2}{p}\right)\B\left(\frac2p,1-\frac2p\right)\\
&=&
\frac{\tilde\sigma^4}{4\kappa^2}
\left(\frac{\kappa}{\tilde\lambda}\right)^{2/p}\left(\frac{2(p^2+2p+2)(p-2)}{p^3(p+1)}\right)\frac{\pi}{\sin(2\pi/p)}.
\end{eqnarray*}

\end{document}